\documentclass{article}

\usepackage[numbers]{natbib}
\usepackage{blindtext}
\usepackage{lineno,hyperref}
\modulolinenumbers[5]
\usepackage{graphicx} 
\usepackage{subfigure}
\usepackage{tikz}
\usepackage{multirow}
\usetikzlibrary{positioning}
\usepackage{colortbl}
\usepackage{amsmath,amssymb}  
\usepackage{chngcntr}
\usepackage{epstopdf}
\usepackage{lipsum}   
\usepackage{setspace}
\usepackage{graphicx}
\usepackage{subcaption}
\usepackage{comment}
\usepackage{nameref}
\usepackage{float}
\usepackage{stackengine}
\usepackage{makecell}
 \usepackage{eurosym}
\usetikzlibrary{bayesnet}
\usepackage{natbib}
\usepackage{xurl}
\usepackage{adjustbox} 
\usepackage{tikz}
\usepackage{pgfplotstable,colortbl}
\usepackage{xcolor}
\pgfplotsset{compat=1.12}
\usepgfplotslibrary[colormaps] 
\usetikzlibrary{pgfplots.colormaps} 
\usepackage{pgfplotstable}
\usepackage{pgfplots}
\usepackage{authblk}
\usepackage[T1]{fontenc}
\usepackage[utf8]{inputenc}
\usepackage{lmodern} 
\usepackage{lineno}
\usepackage{booktabs}
\usepackage{xcolor}
\usepackage{enumitem}
\usepackage{subcaption}
\pgfplotsset{compat=1.17}

\newtheorem{theorem}{Theorem}

\newtheorem{proposition}[theorem]{Proposition}

\newenvironment{proof}[1][Proof]{\noindent\textbf{#1.} }{\ \rule{0.5em}{0.5em}}

\newcommand{\Uad}{\mathcal U_{\mathrm{ad}}}

\newcommand{\dd}{\,\mathrm d}
\begin{document}

\title{Spatially Explicit Optimal Harvesting and Adjoint-Based Diagnostics for Nonlinear Biomass Dynamics}
\author[1]{Tin Nwe Aye}
\author[2]{Wolfgang Bock}
\affil[1]{Department of Mathematics, Linn\'{e}universitetet, Växjö, Sweden}
\affil[2]{Department of Mathematics, Linn\'{e}universitetet, Växjö, Sweden}
\maketitle

\begin{abstract}
We develop a spatially explicit nonlinear biomass model for exploited fish populations incorporating density-dependent Beverton–Holt production, natural mortality, diffusion-driven spatial movement, and harvesting. Starting from a reduced biomass model motivated by spawning stock biomass data, we derive a reaction–diffusion model with Neumann boundary conditions and spatially distributed harvesting. We establish positivity, boundedness, and well-posedness of the biomass dynamics, and identify a critical harvesting threshold separating persistence and extinction regimes of spatially homogeneous equilibria. A distributed optimal harvesting problem is then formulated, and the corresponding state–adjoint optimality system and pointwise projection characterization of the harvesting control are derived. The adjoint variable provides a measure of the marginal future value of biomass within the harvesting objective and offers a diagnostic that complements biomass abundance alone. Numerical simulations examine the coupled evolution of biomass, the adjoint variable, and optimal harvesting effort. The results show that diffusion progressively reduces the initial spatial heterogeneity in biomass, while the computed adjoint and optimal harvesting fields exhibit limited spatial variation under the parameter regime considered. These results illustrate how the coupled state–adjoint–control framework links spatial biomass dynamics with future management value and harvesting decisions.
\end{abstract}

\section{Introduction}
The interaction between nonlinear population dynamics, spatial dispersal, and
resource exploitation leads naturally to a class of reaction-diffusion control
problems in which the state dynamics and the harvesting strategy are strongly
coupled. Such problems arise in fisheries, forestry, epidemiology, and other
renewable-resource systems, where the evolution of a spatially distributed
population is influenced simultaneously by intrinsic growth, mortality,
movement, and externally imposed interventions. From a mathematical point of
view, these models combine nonlinear parabolic dynamics with bounded
distributed controls and raise questions concerning positivity, persistence,
stability, existence of optimal controls, and the structure of the associated
state-adjoint system.

In fisheries applications, these questions are particularly relevant because
harvesting acts on populations that are heterogeneous in space and evolve on
time scales that are not necessarily aligned with the time scale of management
decisions. Fish populations are affected by biological production, natural
mortality, movement, and fishing pressure
\citep{hilborn2013quantitative,quinn1999quantitative,murray2002mathematical,webb1985theory}.
These mechanisms determine the evolution of biomass both globally and locally
\citep{goethel2021revival}. Although fisheries monitoring programmes routinely
monitor aggregate quantities such as spawning stock biomass, fishing mortality,
and recruitment \citep{slu_aqua2025}, such indicators do not directly resolve
the spatial redistribution of biomass or the consequences of spatially
dependent harvesting decisions \citep{goethel2011incorporating}. This motivates
the use of spatially explicit models in which biological dynamics and
harvesting are treated within a common analytical framework.

Reaction-diffusion equations provide a natural mathematical description of
spatial redistribution and dispersal in population systems
\citep{cantrell2004spatial,murray2002mathematical,okubo2001diffusion}. In
combination with nonlinear recruitment and mortality mechanisms, they lead to
semilinear parabolic equations whose qualitative behavior can depend
sensitively on harvesting intensity. Of particular interest is the transition
between persistence and extinction regimes. Such threshold phenomena are
already visible in reduced population models, but their relation to the spatial
problem and to optimal harvesting policies is less immediate. A spatial model
must therefore address not only the existence and boundedness of the state
solution, but also the stability of biologically relevant equilibria and the
effect of dispersal on the long-time behavior of perturbations.

The control problem introduces a second mathematical layer. A harvesting
decision affects both the immediate yield and the future state of the
population, so an optimal policy cannot in general be determined from the
current biomass alone. Optimal control theory provides a systematic framework
for balancing these competing effects
\citep{lions1971optimal,lenhart2007optimal,clark2010mathematical}. For
distributed systems governed by partial differential equations, the
corresponding adjoint variable quantifies the sensitivity of the objective
functional with respect to perturbations of the state
\citep{lions1971optimal,evans2022partial}. In a population-harvesting context,
this gives the adjoint a particularly useful interpretation: it represents the
marginal future value of biomass relative to the optimization objective. The
state and adjoint variables therefore provide complementary information, with
the state describing current abundance and the adjoint measuring the
consequence of preserving or removing biomass for future system performance.

Spatially distributed optimal harvesting problems have been studied in a
variety of settings; see, for example,
\citep{cohen1987review,ding2009optimal,kelly2016optimal}. The present work is
motivated by the observation that the analytical population dynamics, the
critical harvesting threshold, and the adjoint-based interpretation can be
developed within a single nonlinear framework. Rather than treating the
reaction-diffusion model merely as a numerical constraint in an optimization
problem, we first analyze the underlying nonlinear biomass dynamics and then
build the distributed optimal-control problem on this analytical foundation.

More precisely, we consider a biomass model with Beverton-Holt-type
density-dependent production, proportional natural mortality, and harvesting.
We first study a reduced scalar model and identify an explicit critical
harvesting intensity separating the existence and nonexistence of a positive
equilibrium. This threshold provides a reference point for the subsequent
spatial problem. We then introduce a reaction--diffusion model with homogeneous
Neumann boundary conditions and a bounded spatially distributed harvesting
control. For this system we establish positivity and a priori boundedness of
solutions and prove global well-posedness in the natural parabolic setting. We
also analyze the spatially homogeneous equilibria and show that the positive
equilibrium is linearly asymptotically stable whenever it exists, whereas the
zero equilibrium becomes stable above the critical harvesting threshold. Thus,
the persistence threshold of the reduced model is retained in the spatial
system, while diffusion acts to damp nonhomogeneous perturbation modes.

On this basis, we formulate a finite-horizon distributed optimal harvesting
problem with quadratic effort cost. The admissible controls are bounded in
space and time, and the associated state equation is the nonlinear
reaction-diffusion equation described above. We establish the existence of an
optimal distributed control and derive the corresponding first-order
optimality system. The optimal control is characterized pointwise through a
projection formula involving both the biomass state and the adjoint variable.
This characterization makes explicit the competition between current biomass
availability and its future marginal value: harvesting is favored where
biomass is abundant and its adjoint value is comparatively small, whereas large
adjoint values penalize current extraction.

The principal contribution of the paper is therefore the construction of a
single state-adjoint-control framework in which the qualitative dynamics of a
nonlinear biomass equation, the persistence/extinction threshold, and the
distributed harvesting problem are analyzed jointly. In particular, the
adjoint is used not only as an auxiliary variable required to derive the
optimality condition, but also as a diagnostic quantity that measures the
spatio-temporal marginal value of biomass. This provides a natural way to
distinguish regions of high current abundance from regions whose biomass has a
larger contribution to future system performance.

The analytical results are complemented by numerical experiments. As an
empirical motivation, the reduced model is compared with spawning stock biomass
observations, while the spatial optimality system is solved by a
forward-backward sweep method
\citep{mcasey2012convergence,lenhart2007optimal,soukaina2022optimal}. The
simulations track simultaneously the biomass state, the adjoint variable, and
the optimal harvesting effort. For the parameter regime considered, diffusion
progressively suppresses the initial spatial heterogeneity in the biomass, and
the adjoint and harvesting fields likewise approach weakly varying spatial
profiles. A sensitivity study with respect to the diffusion coefficient is used
to examine how strongly the redistribution mechanism affects this
state--adjoint--control coupling. The numerical results thus complement the
analysis by illustrating how spatial smoothing, future marginal value, and
optimal extraction interact within the nonlinear system.

The paper is organized as follows. Section~\ref{sec:biological} introduces the
biological background and empirical motivation. Section~\ref{sec:SPM} presents
the age- and space-structured population model and derives an aggregate
biomass balance. Section~\ref{sec:ReducedM} studies the reduced biomass model
and identifies the critical harvesting threshold.
Section~\ref{sec:SpatialM} develops the spatial reaction-diffusion model and
establishes positivity, boundedness, well-posedness, and stability properties.
Section~\ref{sec:disOptimal} formulates the distributed optimal-control problem
and derives the state--adjoint optimality system.
Section~\ref{sec:numericalSt} describes the numerical method and diagnostic
quantities, while Section~\ref{sec:numericalSi} presents the numerical
experiments. Finally, Section~\ref{sec:discussion} summarizes the main results
and discusses possible extensions.

\section{Biological Background}\label{sec:biological}

This study uses herring (\textit{Clupea harengus}) as the biological context for developing a spatial biomass and harvesting model. The observational data considered in the study concern herring in the Central Baltic Sea, covering subdivisions 25--29 and 32. The available spawning stock biomass (SSB) observations cover 1974--2025, while fishing mortality information for herring aged 3--6 years is available for 1974--2024 \citep{slu_aqua2025}. These data provide an empirical basis for examining the temporal behavior of the reduced biomass model.

Spawning stock biomass represents the biomass of mature individuals contributing to reproduction and is a standard quantity used in fisheries assessment. Recruitment describes the addition of new individuals to the population following the early life stages. In the present study, these quantities are used to motivate a reduced biomass description in which the reproductive contribution of the population is represented through a density-dependent production term.

The detailed age structure of a fish population is important in fisheries assessment, but it is not explicitly represented in the reduced model developed here. Instead, we use a single biomass state variable \(B(t)\) to describe the effective spawning biomass. This reduction provides a tractable formulation for studying the interaction between population dynamics and harvesting while retaining the main biological mechanisms relevant to the present analysis.

We represent density-dependent biomass production using the Beverton--Holt form
\begin{equation}\label{eq:BH}
    R(B)=\frac{aB}{b+B},
\end{equation}
where \(a>0\) determines the maximum production level and \(b>0\) controls the saturation scale. Natural mortality is represented by the proportional loss term \(mB\), where \(m>0\) is the natural mortality rate. Harvesting is represented by
\begin{equation}\label{eq:harvesting}
    H(t)=\alpha u(t)B(t),
\end{equation}
where \(u(t)\) denotes harvesting effort and \(\alpha>0\) represents harvesting efficiency. The resulting reduced biomass dynamics are therefore
\begin{equation}\label{eq:ReducedBiomass}
    \frac{dB}{dt}
=
\frac{aB}{b+B}
-\alpha u(t)B
-mB.
\end{equation}

This formulation is subsequently extended to a spatial setting. When biomass and harvesting vary across a habitat, spatial redistribution is represented by a diffusion term, while harvesting becomes a distributed control \(u(x,t)\). This leads to the spatial reaction--diffusion model studied in the following sections. The spatial formulation is intended as a mathematical representation of biomass redistribution and local harvesting rather than as a movement model of herring.	
	
\section{Structured Population Model}\label{sec:SPM}
To incorporate biological realism and provide a general framework for population dynamics, we consider an age- and space-structured model. This formulation captures spatial movement, ageing processes, and distributed mortality, and serves as a basis for deriving simplified biomass models used in subsequent sections.

Let $B(x,t,a)$ denote the biomass state at spatial location $x \in \Omega$, time $t \in [0,T]$, and age $a \ge 0$. The dynamics are governed by a transport–diffusion–reaction equation:
\begin{align}\label{eq:biomass}
\partial_t B(x,t,a)
+
\partial_a B(x,t,a)
&=
\nabla \cdot \big(D(x)\nabla B(x,t,a)\big)
-
\alpha\,u(x,t)\,B(x,t,a)\\
&\quad
-
d(x,a)\,B(x,t,a) 
\nonumber.
\end{align}
where $D(x)$ denotes the spatial diffusion coefficient. The natural mortality rate and the harvesting control are represented by $d(x,a)$ and $u(x,t)$, respectively. The coefficient $\alpha > 0$ scales the effectiveness of harvesting effort. 

The total biomass is 
\begin{equation}\label{eq:TBiomass}
T(t) = \int_{\Omega} \int_{0}^{\infty} B(x,t,a)\, da\, dx
\end{equation}

Recruitment of newborn individuals is imposed through the boundary condition at age $a=0$:
\begin{equation}\label{eq:recBoundary}
B(x,t,0) = \int_{0}^{\infty} \rho(x,a)\,B(x,t,a)\, da,
\end{equation}
where $\rho(x,a)$ represents the fertility/reproduction or recruitment rate.

\subsection{Total Biomass Balance}
To connect the structured formulation with aggregate population quantities, we consider the evolution of the total biomass. The total biomass at time t is given by Equation
\eqref{eq:TBiomass}.

To determine the rate of change of the total biomass, we differentiate $T(t)$ with respect to time:
\begin{equation*}
\frac{dT}{dt}
=
\int_{\Omega} \int_{0}^{\infty} \partial_t B(x,t,a)\, da\, dx.
\end{equation*}

By using Equation \eqref{eq:biomass},
we rewrite $\partial_t B$ as
\begin{equation*}
\partial_t B
=
- \partial_a B
+ \nabla \cdot (D(x)\nabla B)
- d(x,a)\,B
- \alpha u(x,t)\,B.
\end{equation*}

Substituting into the expression for $\frac{dT}{dt}$ gives
\begin{align*}
\frac{dT}{dt}
&=
\int_{\Omega} \int_{0}^{\infty}
\left(
- \partial_a B
+ \nabla \cdot (D(x)\nabla B)
- d(x,a)\,B
- \alpha u(x,t)\,B
\right)
da\, dx.
\end{align*}

We now evaluate each term separately.
For age-derivative term, we have
\begin{equation*}
\int_{0}^{\infty} -\partial_a B\, da
=
B(x,t,0),
\end{equation*}
where the assumption is used that $B(x,t,a) \to 0$ as $a \to \infty$.

For diffusion term, it is
\begin{equation*}
\int_{\Omega} \nabla \cdot (D(x)\nabla B)\, dx = 0,
\end{equation*}
under no-flux boundary conditions.

Thus, we obtain
\begin{equation*}
\frac{dT}{dt}
=
\int_{\Omega} B(x,t,0)\, dx
-
\int_{\Omega} \int_{0}^{\infty} d(x,a)\,B\, da\, dx
-
\alpha\int_{\Omega} \int_{0}^{\infty} u(x,t)\,B\, da\, dx.
\end{equation*}

Using the recruitment boundary condition \eqref{eq:recBoundary},
we finally obtain
\begin{align*}
\frac{dT}{dt} &=
\int_{\Omega} \int_{0}^{\infty} \rho(x,a)\,B(x,t,a)\, da\, dx
-
\int_{\Omega} \int_{0}^{\infty} d(x,a)\,B(x,t,a)\, da\, dx\\
&\qquad-
\alpha\int_{\Omega} \int_{0}^{\infty} u(x,t)\,B(x,t,a)\, da\, dx.
\end{align*}

To simplify the notation, we introduce the aggregated quantities
\begin{align*}
W(t) &= \int_{\Omega}\int_{0}^{\infty} \rho(x,a)\,B(x,t,a)\, da\, dx, \\
M(t) &= \int_{\Omega}\int_{0}^{\infty} d(x,a)\,B(x,t,a)\, da\, dx, \\
U(t) &= \alpha\int_{\Omega}\int_{0}^{\infty} u(x,t)\,B(x,t,a)\, da\, dx.
\end{align*}

The total biomass dynamics can then be written as
\begin{equation}\label{eq:aggregate_balance}
\frac{dT}{dt} = W(t) - M(t) - U(t).
\end{equation}

Equation~\eqref{eq:aggregate_balance} shows that diffusion redistributes biomass in space but does not directly change total biomass under no-flux boundary conditions. The total biomass is changed by recruitment, natural mortality, and harvesting. This observation motivates the reduced biomass model studied in the next section.

\section{Reduced Biomass Model}\label{sec:ReducedM}
The aggregate balance derived in the previous section describes the evolution of total biomass and motivates the construction of a reduced, one-dimensional closure. Since recruitment is primarily determined by the reproductive capacity of the mature population, spawning stock biomass is selected as the reduced observable rather than total biomass. 

To obtain a tractable one-patch model, we assume that the age and spatial structure remain sufficiently stable so that spawning stock biomass can be represented by an effective scalar variable $B(t)$, proportional to the corresponding aggregate biomass. The structured recruitment term is then approximated by a density-dependent recruitment function of the effective spawning stock biomass. 
In this reduced closure, we adopt the Beverton-Holt form \eqref{eq:BH}. Harvesting is assumed proportional to both effort and biomass according to \eqref{eq:harvesting}, yielding the reduced dynamics \eqref{eq:ReducedBiomass}.
\begin{proposition}[Positivity and boundedness]
Let $u\in L^\infty(0,T)$ satisfy
\[
0\le u(t)\le u_{\max}
\qquad\text{a.e. on }(0,T).
\]
If $B_0\ge0$, then the initial value problem
\begin{equation}\label{eq:reduced_model}
\frac{\dd B}{\dd t}
=
\frac{aB}{b+B}
-\alpha u(t)B
-mB,
\qquad
B(0)=B_0,
\end{equation}
admits a unique global nonnegative solution on $[0,T]$. Moreover,
\[
B(t)\ge0
\qquad\text{for all } t\in[0,T],
\]
and
\[
0\le B(t)\le
\max\left\{B_0,\frac{a}{m}\right\}
\qquad\text{for all } t\in[0,T].
\]
\end{proposition}
\begin{proof}
The right-hand side of \eqref{eq:reduced_model} is measurable in $t$ and
locally Lipschitz continuous in $B$. Hence a unique local solution exists.
From Equation \eqref{eq:ReducedBiomass}, $\dot B=0$ whenever $B=0$.
Therefore, the nonnegative half-line
is positively invariant, and thus $B(t)\ge0$ for all $t\in[0,T]$.

Furthermore, using $u(t)\ge0$ and
\[
\frac{aB}{b+B}\le a,
\]
we obtain
\[
\frac{dB}{dt}\le a-mB.
\]
A comparison argument with the solution of
\[
y'=a-my, \qquad y(0)=B_0,
\]
yields
\[
B(t)\le \max\left\{B_0,\frac{a}{m}\right\}
\qquad \text{for all } t\in[0,T].
\]
The boundedness of $B$ excludes finite-time blow-up and therefore the
local solution extends globally.
\end{proof}

\begin{proposition}[Equilibria and recovery threshold]
Consider \eqref{eq:reduced_model} under a constant harvesting effort
$u(t)\equiv u_c\ge0$. Then the equilibrium points are given by
\[
B^*_0=0
\]
and
\[
B^*_1(u_c)=\frac{a}{m+\alpha u_c}-b.
\]

Moreover, a positive equilibrium, $B^*_1(u_c)>0$, exists if and only if
\[
u_c<u_{\mathrm{crit}},
\qquad
u_{\mathrm{crit}}
:=
\frac{1}{\alpha}\left(\frac{a}{b}-m\right).
\]
\end{proposition}

\begin{proof}
Equilibria satisfy
\[
0=\frac{aB}{b+B}-(m+\alpha u_c)B.
\]
Factoring out $B$ gives
\[
B\left(\frac{a}{b+B}-(m+\alpha u_c)\right)=0.
\]
Hence $B^*_0=0$ and
\[
B^*_1(u_c)=\frac{a}{m+\alpha u_c}-b.
\]
The condition $B^*_1(u_c)>0$ is equivalent to
\[
a>b(m+\alpha u_c),
\]
or, equivalently,
\[
u_c<\frac{1}{\alpha}\left(\frac{a}{b}-m\right)
=u_{\mathrm{crit}}.
\]
The conclusion follows.
\end{proof}

\subsection{Optimal Harvesting in the Reduced Model}
	
Before considering the spatially distributed problem, we first formulate the optimal harvesting problem for the reduced biomass model. We define the admissible control set as
\[\Uad^{\mathrm{ODE}}=
\left\{u\in L^\infty(0,T):0\le u(t)\le u_{\max}\ \text{a.e.}\right\}.\]
A basic yield-maximization problem is
\begin{equation}\label{eq:ode_objective}
\max_{u\in\Uad^{\mathrm{ODE}}}J(u)=
\int_0^T \alpha u(t)B(t)\,\dd t,
\end{equation}
subject to~\eqref{eq:reduced_model}.
	
The Hamiltonian is
\begin{equation}\label{eq:H_ode}
\mathcal H(B,u,\lambda)=\alpha uB+\lambda
\left(\frac{aB}{b+B}-\alpha uB-mB\right).
\end{equation}

The adjoint equation is
\begin{equation}\label{eq:adjoint_ode}
\dot\lambda=-\frac{\partial\mathcal H}{\partial B}=
-\alpha u-\lambda\left(\frac{ab}{(b+B)^2}-\alpha u-m\right),\qquad\lambda(T)=0.
\end{equation}
The switching function is
\begin{equation}\label{eq:switching_ode}
\Phi(t)=\frac{\partial\mathcal H}{\partial u}=\alpha B(t)(1-\lambda(t)).
\end{equation}
Hence, whenever $B(t)>0$, the optimal control satisfies
\begin{equation}\label{eq:bang_bang_ode}
u^\ast(t)=\begin{cases}
u_{\max}, & \lambda(t)<1,\\[0.3em]
0, & \lambda(t)>1.
\end{cases}
\end{equation}

The bang--bang structure of \eqref{eq:bang_bang_ode} follows from the linear dependence of both the state equation and the objective functional on the
control variable. Although mathematically natural, such controls may be numerically sensitive and do not penalize excessive harvesting effort. For this reason, the spatially distributed model introduced below incorporates a quadratic effort cost, leading to a projected feedback characterization of the optimal control rather than a purely bang--bang strategy.

\section{Spatial Reaction--Diffusion Biomass Model}\label{sec:SpatialM}
We now extend the reduced biomass model to a spatially distributed setting. The bounded domain \(\Omega\subset\mathbb{R}^d\) is assumed to have a sufficiently smooth boundary \(\partial\Omega\), where \(d\in\{1,2\}\). The biomass state is denoted by $B(x,t)$, where $x\in\Omega$ and $t\in[0,T]$. Spatial movement is represented by diffusion, while recruitment, mortality, and harvesting act locally.

The spatial biomass equation is
\begin{equation}\label{eq:spatial_model}
\partial_t B=D\Delta B+\frac{aB}{b+B}-\alpha u(x,t)B-mB,
\qquad x\in\Omega,\quad 0<t<T,
\end{equation}
with no-flux boundary condition
\begin{equation}\label{eq:neumann}
\partial_\nu B=0,
\qquad x\in\partial\Omega,\quad 0<t<T,
\end{equation}
and initial condition
\begin{equation}\label{eq:initial_spatial}
B(x,0)=B_0(x)\ge0.
\end{equation}
Here $D>0$ is the diffusion coefficient and $u(x,t)$ is a spatially distributed harvesting effort satisfying
\[0\le u(x,t)\le u_{\max}.\]

\begin{theorem}[Positivity and boundedness]\label{thm:positivity}
Assume that $B_0\in L^\infty(\Omega)$, $B_0\ge0$, and
$u\in L^\infty((0,T)\times\Omega)$ satisfies
\[
0\le u(x,t)\le u_{\max}
\qquad\text{a.e. in }\Omega\times(0,T).
\]
Then every classical solution of
\eqref{eq:spatial_model}--\eqref{eq:initial_spatial} satisfies
\[
B(x,t)\ge0
\qquad\text{for all }x\in\Omega,\; t\in[0,T].
\]
Moreover,
\[
0\le B(x,t)
\le
\|B_0\|_{L^\infty(\Omega)}+aT
\qquad\text{for all }x\in\Omega,\; t\in[0,T].
\]
\end{theorem}

\begin{proof}
Let
\[
F(x,t,B)
=
\frac{aB}{b+B}
-\alpha u(x,t)B
-mB.
\]
Since
\[
F(x,t,0)=0,
\]
the zero function is a subsolution of \eqref{eq:spatial_model}.
Because
\[
0=B_{\mathrm{sub}}(x,0)\le B_0(x),
\]
the parabolic comparison principle implies
\[
B(x,t)\ge0
\qquad\text{for all }(x,t)\in\Omega\times[0,T].
\]

Furthermore, using $B\ge0$,
\[
\frac{aB}{b+B}\le a,
\qquad
-\alpha u(x,t)B\le0,
\qquad
-mB\le0.
\]
From Equation \eqref{eq:spatial_model}, we have
\[
\partial_t B-D\Delta B\le a.
\]

Let $y(t)$ solve
\[
y'(t)=a,
\qquad
y(0)=\|B_0\|_{L^\infty(\Omega)}.
\]
Then
\[
y(t)=\|B_0\|_{L^\infty(\Omega)}+at.
\]
Since $y$ is independent of $x$,
\[
\Delta y=0,
\]
and therefore
\[
\partial_t y-D\Delta y=a.
\]
Moreover,
\[
B_0(x)\le \|B_0\|_{L^\infty(\Omega)}
= y(0)
\qquad\text{for all }x\in\Omega.
\]
Applying the parabolic comparison principle gives
\[
B(x,t)\le y(t)
=
\|B_0\|_{L^\infty(\Omega)}+at.
\]
Since $t\le T$,
\[
B(x,t)
\le
\|B_0\|_{L^\infty(\Omega)}+aT.
\]
This completes the proof.
\end{proof}

\begin{theorem}[Existence and uniqueness]
Assume that $B_0\in L^\infty(\Omega)$, $B_0\ge0$, and
\[
u\in L^\infty((0,T)\times\Omega),
\qquad
0\le u(x,t)\le u_{\max}
\]
almost everywhere. Then the problem
\eqref{eq:spatial_model}--\eqref{eq:initial_spatial}
admits a unique global mild solution
\[
B\in C([0,T];L^2(\Omega))
\cap
L^2(0,T;H^1(\Omega)).
\]
Furthermore, if $B_0$ and $\partial\Omega$ are sufficiently regular,
then the solution is classical for $t>0$.
\end{theorem}

\begin{proof}
Consider the operator
\[
A:=D\Delta
\]
subject to the homogeneous Neumann boundary condition
\[
\partial_\nu v=0
\quad\text{on }\partial\Omega.
\]
Its domain is
\[
\mathcal D(A)
=
\{v\in H^2(\Omega):\partial_\nu v=0
\text{ on }\partial\Omega\}.
\]
It is well known that $A$ generates an analytic semigroup
$\{e^{tA}\}_{t\ge0}$ on $L^2(\Omega)$.

The difficulty in applying the standard semilinear theory directly to
the biological reaction term is that the Beverton--Holt function
\[
B\mapsto \frac{aB}{b+B}
\]
is singular at $B=-b$. We therefore introduce a globally defined
extension of the reaction term.

Let
\[
R:=\|B_0\|_{L^\infty(\Omega)}+aT
\]
and choose a smooth function
$\rho:\mathbb R\to[0,1]$ such that
\[
\rho(B)=1
\quad\text{for }B\ge0,
\qquad
\rho(B)=0
\quad\text{for }B\le-\frac{b}{2}.
\]
Define the truncated recruitment function by
\[
\widetilde{G}(B)
=
\rho(B)\frac{aB}{b+B}.
\]
Since the support of $\rho$ is separated from the singular point $B=-b$, the function $\widetilde{G}$ is continuously differentiable and globally Lipschitz on $\mathbb R$. Consequently, for
\[
\widetilde F(x,t,B)
=
\widetilde G(B)
-\alpha u(x,t)B
-mB,
\]
we have
\[
|\widetilde F(x,t,B_1)
-\widetilde F(x,t,B_2)|
\le C|B_1-B_2|
\]
for all $B_1,B_2\in\mathbb R$, with a constant $C>0$ independent of $(x,t)$, since
$0\le u(x,t)\le u_{\max}$ almost everywhere.

We therefore consider the extended problem
\[
\partial_t B=AB+\widetilde F(x,t,B),
\]
with the same initial and boundary conditions as
\eqref{eq:spatial_model}--\eqref{eq:initial_spatial}. Its mild formulation is
\[
B(t)
=
e^{tA}B_0
+
\int_0^t
e^{(t-s)A}
\widetilde F(\cdot,s,B(s))
\,ds .
\]
Since $\widetilde F$ is globally Lipschitz in $B$, standard semilinear parabolic theory yields a unique global mild solution
\[
B\in C([0,T];L^2(\Omega))
\cap
L^2(0,T;H^1(\Omega)).
\]

It remains to verify that this solution also solves the original biological problem. Since $B_0\ge0$ and the extended reaction satisfies
\[
\widetilde F(x,t,0)=0,
\]
the comparison principle gives
\[
B(x,t)\ge0
\]
almost everywhere in $\Omega\times(0,T)$. Hence, throughout the solution trajectory,
\[
\rho(B(x,t))=1,
\]
and therefore
\[
\widetilde G(B(x,t))
=
\frac{aB(x,t)}{b+B(x,t)}.
\]
Thus the extended equation coincides with the original equation on the
nonnegative region, and the obtained mild solution is a mild solution
of \eqref{eq:spatial_model}--\eqref{eq:initial_spatial}.

Moreover, by Theorem~\ref{thm:positivity},
\[
0\le B(x,t)
\le
\|B_0\|_{L^\infty(\Omega)}+aT
\]
for all $t\in[0,T]$. Hence the solution remains uniformly bounded on
the entire time interval and cannot exhibit finite-time blow-up.

Finally, the smoothing property of the analytic semigroup implies that, whenever $B_0$ and $\partial\Omega$ possess sufficient regularity, the mild solution becomes classical for every $t>0$.
This completes the proof.
\end{proof}

\subsection{Spatially homogeneous equilibria and recovery threshold}

Assume that the harvesting effort is spatially constant,
\[
u(x,t)\equiv u_c.
\]
A spatially homogeneous equilibrium is a steady state of the form
\[
B(x,t)\equiv B^*,
\]
where $B^*$ is independent of space and time. Since $B^*$ is constant,
\[
\Delta B^*=0,
\]
and substituting into \eqref{eq:spatial_model} yields
\[
0
=
\frac{aB^*}{b+B^*}
-
(m+\alpha u_c)B^*.
\]

Factoring out $B^*$ gives
\[
B^*
\left(
\frac{a}{b+B^*}
-(m+\alpha u_c)
\right)
=0.
\]

Hence $B^*=0$ is always an equilibrium. A positive equilibrium satisfies
\[
\frac{a}{b+B^*}
=
m+\alpha u_c,
\]
and therefore
\[
B^*(u_c)
=
\frac{a}{m+\alpha u_c}
-b.
\]
This equilibrium exists precisely when
\[
B^*(u_c)>0,
\]
or equivalently,
\[
a>b(m+\alpha u_c).
\]
Solving for $u_c$ gives the critical harvesting threshold
\begin{equation*}
u_c
<
u_{\rm crit}
:=
\frac1\alpha
\left(
\frac ab-m
\right).
\end{equation*}

Thus the same recovery threshold appears as in the reduced biomass model.
The reason is that for homogeneous equilibria the diffusion term vanishes
identically and therefore does not affect the equilibrium relation.

\begin{proposition}[Linear stability of the positive homogeneous equilibrium]
Assume that $u(x,t)\equiv u_c$ and that a positive homogeneous equilibrium $B^*=B^*(u_c)>0$ exists. Then $B^*$ is linearly asymptotically stable under homogeneous Neumann boundary conditions.
Moreover, the zero equilibrium is linearly asymptotically stable when
$u_c>u_{\mathrm{crit}}$, where
\begin{equation}\label{eq:ucrit_spatial}
    u_{\mathrm{crit}}
=
\frac{1}{\alpha}
\left(\frac{a}{b}-m\right).
\end{equation}

\end{proposition}

\begin{proof}
For a spatially homogeneous equilibrium under the constant control
$u_c$, we have
\[
f(B^*)=0.
\]
Since $B^*>0$, it follows that
\[
\frac{a}{b+B^*}
=
m+\alpha u_c
=:q.
\]
Differentiating $f$ gives
\[
f'(B)
=
\frac{ab}{(b+B)^2}-q.
\]
Evaluating at $B^*$ and using
\[
\frac{a}{b+B^*}=q,
\]
we obtain
\[
\begin{aligned}
f'(B^*)
&=
\frac{ab}{(b+B^*)^2}-q\\
&=
q\left(
\frac{b}{b+B^*}-1
\right)\\
&=
-q\frac{B^*}{b+B^*}.
\end{aligned}
\]
Since $q>0$, $B^*>0$, and $b+B^*>0$, we conclude that
\[
f'(B^*)<0.
\]

To study the spatial stability, consider a small perturbation of the equilibrium,
\[
B(x,t)=B^*+z(x,t).
\]
Since $B^*$ is spatially homogeneous,
\[
\Delta B^*=0,
\qquad
f(B^*)=0.
\]
Substituting $B=B^*+z$ into the spatial model gives
\[
\partial_t z
=
D\Delta z+f(B^*+z).
\]
Using the first-order Taylor expansion
\[
f(B^*+z)
=
f(B^*)+f'(B^*)z+O(z^2),
\]
and neglecting the higher-order terms, we obtain the linearized problem
\[
\partial_t z
=
D\Delta z+f'(B^*)z,
\qquad
\partial_\nu z=0.
\]

We expand the perturbation in eigenfunctions of the Neumann Laplacian.
Let $\{\phi_k\}_{k\ge0}$ satisfy
\[
-\Delta\phi_k=\mu_k\phi_k,
\qquad
\partial_\nu\phi_k=0,
\]
with eigenvalues
\[
0=\mu_0\le\mu_1\le\mu_2\le\cdots.
\]
Since the eigenfunctions form an orthogonal basis of $L^2(\Omega)$, the perturbation can be represented as
\[
z(x,t)
=
\sum_{k=0}^{\infty}c_k(t)\phi_k(x).
\]
Substituting this expansion into the linearized equation and using
\[
\Delta\phi_k=-\mu_k\phi_k
\]
yields
\[
c_k'(t)
=
\bigl(f'(B^*)-D\mu_k\bigr)c_k(t).
\]
Hence, the growth rate associated with the $k$th spatial mode is
\[
\lambda_k
=
f'(B^*)-D\mu_k.
\]
Since $\mu_k\ge0$ and $D\ge0$,
\[
\lambda_k
\le
f'(B^*)
<0.
\]
Therefore, every spatial perturbation mode decays exponentially in time, and
\[
z(x,t)\to0
\qquad\text{as }t\to\infty.
\]
Thus, the positive homogeneous equilibrium $B^*$ is linearly asymptotically stable under homogeneous Neumann boundary conditions.

We next consider the zero equilibrium $B=0$, which satisfies $f(0)=0$. The corresponding linearized growth rate is
\[
f'(0)
=
\frac{a}{b}-m-\alpha u_c.
\]
We define the critical harvesting level by Equation \eqref{eq:ucrit_spatial}.
Then
\[
f'(0)
=
\alpha
\left(
u_{\mathrm{crit}}-u_c
\right).
\]

If $u_c>u_{\mathrm{crit}}$, then we get
$f'(0)<0$.
For the linearized spatial problem around $B=0$, the modal growth rates are
\[
f'(0)-D\mu_k.
\]
Since $\mu_k\ge0$, all these growth rates are negative. Hence the zero equilibrium is linearly asymptotically stable.

If $u_c<u_{\mathrm{crit}}$,
then we obtain $f'(0)>0$,
so the spatially homogeneous mode corresponding to $\mu_0=0$ grows
exponentially. Thus, the zero equilibrium is linearly unstable. In this regime, the positive homogeneous equilibrium exists and, as shown above, is linearly asymptotically stable.

At the critical value $u_c=u_{\mathrm{crit}}$,
we have $f'(0)=0$, so linearization alone does not determine the stability of the zero
equilibrium.

Therefore, $u_{\mathrm{crit}}$ identifies the transition between the persistence and extinction regimes of the spatially homogeneous equilibria.
\end{proof}

\section{Distributed Optimal Harvesting Problem}\label{sec:disOptimal}

We now formulate an optimal harvesting problem for the spatially distributed biomass model. In contrast to the reduced model, the harvesting effort may vary
both in space and time and is therefore represented by a distributed control
\[
u=u(x,t).
\]
This allows different harvesting intensities at different locations of the habitat.

To account for the economic cost of harvesting, we introduce a quadratic penalty on the control. Such a term is biologically reasonable since increased
fishing effort requires additional resources, and it is mathematically useful because it leads to a convex control cost and a stable characterization of the optimal control.

The admissible control set is defined by
\[
U_{\mathrm{ad}}
=
\left\{
u\in L^\infty((0,T)\times\Omega):
0\le u(x,t)\le u_{\max}
\ \text{a.e.}
\right\}.
\]

For $\gamma>0$, we consider the objective functional
\begin{equation}\label{eq:objective_spatial}
J(u)
=
\int_0^T
\int_\Omega
\left(
\alpha u(x,t)B(x,t)
-
\frac{\gamma}{2}u(x,t)^2
\right)
\,dx\,dt.
\end{equation}

The first term represents the harvested biomass, while the second term measures the cost associated with harvesting effort. The optimization problem is
\begin{equation}\label{eq:optimal_control_spatial}
\max_{u\in U_{\mathrm{ad}}}
J(u)
\quad
\text{subject to}
\quad
\eqref{eq:spatial_model}
-
\eqref{eq:initial_spatial}.
\end{equation}

For every admissible harvesting strategy $u\in U_{\mathrm{ad}}$, the state
equation determines a corresponding biomass distribution $B(x,t)$. Therefore the control cannot be chosen independently of the state dynamics, and the
optimization is carried out over all admissible spatio--temporal harvesting policies.
\newpage
\begin{theorem}[Existence of an optimal distributed control]
\label{the:optimal}

Assume that $B_0\in L^\infty(\Omega)$, $B_0\ge0$, and
$U_{\mathrm{ad}}\neq\emptyset$. Then there exists at least one control $u^*\in U_{\mathrm{ad}}$ maximizing \eqref{eq:objective_spatial} subject to
the state equation \eqref{eq:spatial_model}.

\end{theorem}

\begin{proof}

Let $(u_n)$ be a maximizing sequence in $U_{\mathrm{ad}}$, namely
\[
J(u_n)
\longrightarrow
\sup_{u\in U_{\mathrm{ad}}} J(u).
\]

Since
\[
0\le u_n(x,t)\le u_{\max}
\qquad\text{a.e. in }\Omega\times(0,T),
\]
the sequence $(u_n)$ is uniformly bounded in
$L^\infty((0,T)\times\Omega)$.

By the Banach--Alaoglu theorem, there exist a subsequence, still denoted by
$(u_n)$, and a function $u^*\in L^\infty((0,T)\times\Omega)$ such that
\[
u_n \rightharpoonup^* u^*
\qquad\text{in }L^\infty((0,T)\times\Omega).
\]
Here, $\rightharpoonup^*$ denotes weak-* convergence in
$L^\infty((0,T)\times\Omega)$.

We now verify that $u^*\in U_{\mathrm{ad}}$. Let
$\varphi\in L^1((0,T)\times\Omega)$ be nonnegative. Since
$u_n\ge0$ almost everywhere,
\[
\int_0^T\int_\Omega u_n\varphi\,dx\,dt\ge0.
\]
Passing to the weak-* limit gives
\[
\int_0^T\int_\Omega u^*\varphi\,dx\,dt\ge0.
\]
Hence $u^*\ge0$ almost everywhere. Similarly, since
$u_{\max}-u_n\ge0$ almost everywhere,
\[
\int_0^T\int_\Omega
(u_{\max}-u_n)\varphi\,dx\,dt\ge0.
\]
Passing to the limit yields
\[
\int_0^T\int_\Omega
(u_{\max}-u^*)\varphi\,dx\,dt\ge0.
\]
Therefore $u^*\le u_{\max}$ almost everywhere. Thus
\[
0\le u^*\le u_{\max}
\qquad\text{a.e. in }\Omega\times(0,T),
\]
and consequently $u^*\in U_{\mathrm{ad}}$.

For each $u_n$, let $B_n$ be the corresponding state solution of
\eqref{eq:spatial_model}. By Theorem~\ref{thm:positivity},
\[
0\le B_n(x,t)\le C
\]
for some constant $C>0$ independent of $n$. Consequently,
\[
\|B_n\|_{L^\infty((0,T)\times\Omega)}
\le C.
\]

We define
\[
F_n
=
\frac{aB_n}{b+B_n}
-\alpha u_nB_n
-mB_n .
\]

Multiplying \eqref{eq:spatial_model} by $B_n$ and integrating over $\Omega$, we obtain
\[
\frac12\frac{d}{dt}
\|B_n(t)\|_{L^2(\Omega)}^2
+
D\|\nabla B_n(t)\|_{L^2(\Omega)}^2
=
\int_\Omega F_n B_n\,dx.
\]

Since $(u_n)$ and $(B_n)$ are uniformly bounded in
$L^\infty((0,T)\times\Omega)$, there exists a constant $C>0$ such that
\[
|F_n|\le C.
\]
Hence
\[
\left|
\int_\Omega F_n B_n\,dx
\right|
\le C|\Omega|.
\]

Integrating with respect to time over $(0,T)$ yields
\[
\sup_{t\in[0,T]}
\|B_n(t)\|_{L^2(\Omega)}^2
+
D\int_0^T
\|\nabla B_n(t)\|_{L^2(\Omega)}^2\,dt
\le C,
\]
where $C$ is independent of $n$.

Therefore
\[
(B_n)
\quad\text{is bounded in}\quad
L^2(0,T;H^1(\Omega)).
\]

Furthermore, for every $\varphi\in H^1(\Omega)$,
Green's formula and the homogeneous Neumann boundary condition imply
\[
\langle \Delta B_n,\varphi\rangle
=
-\int_\Omega
\nabla B_n\cdot\nabla\varphi\,dx.
\]

Hence
\[
|\langle \Delta B_n,\varphi\rangle|
\le
\|\nabla B_n\|_{L^2(\Omega)}
\|\nabla\varphi\|_{L^2(\Omega)},
\]
which gives
\[
\|\Delta B_n\|_{H^{-1}(\Omega)}
\le
\|\nabla B_n\|_{L^2(\Omega)}.
\]

Consequently,
\[
(\Delta B_n)
\quad\text{is bounded in}\quad
L^2(0,T;H^{-1}(\Omega)).
\]

Since $(F_n)$ is uniformly bounded in
$L^\infty((0,T)\times\Omega)$, it is bounded in
$L^2(0,T;H^{-1}(\Omega))$. From
\[
\partial_t B_n
=
D\Delta B_n
+
F_n,
\]
it follows that $(\partial_t B_n)$ is bounded in
$L^2(0,T;H^{-1}(\Omega))$.

The Aubin--Lions compactness lemma therefore yields the existence of a
subsequence, still denoted by $(B_n)$, and a function $B^*$ such that
\[
B_n \to B^*
\qquad\text{strongly in }
L^2((0,T)\times\Omega).
\]

Moreover, up to a further subsequence,
\[
B_n(x,t)\to B^*(x,t)
\qquad\text{a.e. in }\Omega\times(0,T).
\]

Since
\[
u_n \rightharpoonup^* u^*
\quad\text{in }L^\infty((0,T)\times\Omega)
\]
and
\[
B_n \to B^*
\quad\text{strongly in }L^2((0,T)\times\Omega),
\]
it follows that
\[
u_n B_n
\rightharpoonup
u^*B^*
\qquad\text{weakly in }L^2((0,T)\times\Omega).
\]

Furthermore, the function $B\mapsto \frac{aB}{b+B}$
is continuous on the bounded interval
$0\le B\le C$,
and therefore
\[
\frac{aB_n}{b+B_n}
\to
\frac{aB^*}{b+B^*}
\qquad\text{strongly in }L^2((0,T)\times\Omega).
\]

Passing to the limit in the weak formulation of
\eqref{eq:spatial_model}, we conclude that $B^*$ satisfies the state equation associated with the control $u^*$. Hence $(u^*,B^*)$ is an admissible control--state pair.

It remains to show that $u^*$ attains the maximal value of the objective functional $J$. To this end, we use the convergence properties established above together with the weak lower semicontinuity of the $L^2$-norm.

Moreover, since $u_n\rightharpoonup^* u^*$ weakly in
$L^\infty((0,T)\times\Omega)$, it also follows that
$u_n\rightharpoonup u^*$ weakly in
$L^2((0,T)\times\Omega)$. Thus, the weak lower semicontinuity of the
$L^2$-norm yields
\[
\|u^*\|_{L^2((0,T)\times\Omega)}^2
\le
\liminf_{n\to\infty}
\|u_n\|_{L^2((0,T)\times\Omega)}^2.
\]

Hence,
\[
-\frac{\gamma}{2}
\|u^*\|_{L^2((0,T)\times\Omega)}^2
\ge
\limsup_{n\to\infty}
\left(
-\frac{\gamma}{2}
\|u_n\|_{L^2((0,T)\times\Omega)}^2
\right).
\]

Combining the convergence of the benefit term with the above estimate for the harvesting-cost term, we obtain
\[
J(u^*)
\ge
\limsup_{n\to\infty} J(u_n).
\]

Since $(u_n)$ is a maximizing sequence,
\[
\limsup_{n\to\infty} J(u_n)
=
\sup_{u\in U_{\mathrm{ad}}} J(u).
\]

It follows that
\[
J(u^*)
=
\sup_{u\in U_{\mathrm{ad}}} J(u),
\]
and $u^*$ is an optimal distributed harvesting control.

\end{proof}

\subsection{Adjoint Equation and Optimality Condition}\label{sec:adjAndOpt}

We now derive the optimality system associated with the distributed harvesting problem. Let $(u^*,B^*)$ be an optimal control--state pair obtained in Theorem~\ref{the:optimal}.

The Hamiltonian density is
\begin{equation}\label{eq:H_pde}
\mathcal H(B,u,p)
=
\alpha uB
-
\frac{\gamma}{2}u^2
+
p
\left(
D\Delta B
+
\frac{aB}{b+B}
-
\alpha uB
-
mB
\right).
\end{equation}

Analogously to the reduced model, the adjoint variable satisfies
\[
-\partial_t p
=
\frac{\partial \mathcal H}{\partial B},
\]
where the diffusion term is transferred from $B$ to $p$ by integration by parts. Consequently, the adjoint equation is
\begin{equation}\label{eq:adjoint_pde}
-\partial_t p
-
D\Delta p
=
\frac{ab}{(b+B^*)^2}p
-
mp
+
\alpha u^*(1-p)
\qquad \text{in } \Omega\times(0,T),
\end{equation}
with terminal condition
\begin{equation}\label{eq:adjoint_terminal_pde}
p(x,T)=0,
\qquad x\in\Omega,
\end{equation}
and homogeneous Neumann boundary condition
\begin{equation}\label{eq:adjoint_bc_pde}
\partial_\nu p=0
\qquad \text{on }\partial\Omega\times(0,T).
\end{equation}

Differentiation of \eqref{eq:H_pde} with respect to the control variable gives
\begin{equation}\label{eq:hamiltonian_derivative}
\frac{\partial\mathcal H}{\partial u}
=
\alpha B
-
\gamma u
-
\alpha Bp
=
\alpha B(1-p)-\gamma u.
\end{equation}

Hence the unconstrained maximizer is
\begin{equation}\label{eq:unew}
    u=\frac{\alpha B(1-p)}{\gamma}.
\end{equation}

Since admissible controls satisfy
\[
0\le u\le u_{\max},
\]
the optimal control is characterized by
\begin{equation}\label{eq:optimal_control_pde}
u^*(x,t)
=
\operatorname{Proj}_{[0,u_{\max}]}
\left(
\frac{\alpha B^*(x,t)\bigl(1-p(x,t)\bigr)}
{\gamma}
\right).
\end{equation}

Formula \eqref{eq:optimal_control_pde} shows that harvesting is increased in regions where biomass state is high and the adjoint value is low. Conversely, harvesting is reduced in regions where the adjoint variable is large, indicating that the biomass has a high marginal future value. Thus the adjoint variable identifies regions where biomass should be conserved because of its contribution to future stock
development.

\section{Numerical Strategy}\label{sec:numericalSt}

The purpose of the numerical experiments is to illustrate the spatial dynamics of the biomass distribution and to compare different harvesting policies. 
In addition to the optimal control derived in Section~\ref{sec:adjAndOpt}, we consider a data-driven harvesting strategy based on
observational information; see Figure~2 in \citealp{slu_aqua2025}.

\subsection{Harvesting Strategies}\label{sec:harvest}
To illustrate the harvesting mechanisms considered in the numerical study, we introduce the harvesting policies that will be used throughout the simulations. These include the data-driven harvesting policy derived from observational information and the harvesting policy obtained from the optimal control framework.

The numerical results below are obtained by solving the state--adjoint system and the associated first-order optimality condition. Thus, the computed control represents a numerical solution of the optimality system.

The data-driven harvesting policy introduced in the reduced-model analysis is represented by
\begin{equation}\label{eq:uniform_control}
u(x,t)=u_{\mathrm{obs}}(t),
\end{equation}
where $u_{\mathrm{obs}}(t)$ denotes the harvesting effort reconstructed
from the observational data.

The dataset used in this study consists of spawning stock biomass (SSB) observations. Spawning stock biomass denotes the total biomass of mature individuals contributing to reproduction and is reported in tonnes. The available records cover the period 1974--2025 (see Figure~2 in \citealp{slu_aqua2025}).

The optimal control framework developed in this work seeks a harvesting policy that balances harvesting benefits with the long-term value of the biomass stock. The resulting harvesting policy is characterized by the Equation \eqref{eq:optimal_control_pde}.
Here, the adjoint variable $p$ measures the marginal future value of biomass and therefore influences whether harvesting should be increased or reduced at a given location and time.

The numerical experiments compare the biomass dynamics associated with the data-driven harvesting policy and the model-based optimal harvesting policy. Particular attention is given to the role of the adjoint variable in identifying regions where biomass conservation may
yield greater long-term benefits than immediate harvesting.

\subsection{Numerical Implementation and Diagnostic Quantities}

The numerical study is organized in two stages. First, the reduced model is compared with the available spawning stock biomass (SSB) observations to assess its ability to reproduce the observed biomass
dynamics. Subsequently, the spatially distributed model is used to investigate the effects of diffusion, harvesting policies, and optimal control. The harvesting policies introduced above are employed throughout the simulations via
\eqref{eq:uniform_control} and
\eqref{eq:optimal_control_pde} in Section~\ref{sec:harvest}.

\subsubsection{Biomass and harvesting normalization.}
The SSB observations used for the model comparison are expressed relative to the reference spawning biomass. Accordingly, biomass is represented
on a normalized scale and is denoted by
\[
B(t)
=
\frac{B_{\mathrm{phys}}(t)}{B_{\mathrm{ref}}},
\]
where \(B_{\mathrm{phys}}(t)\) is the corresponding physical biomass and \(B_{\mathrm{ref}}\) is the reference biomass used for normalization. 
The harvesting input is likewise expressed relative to the reference fishing mortality. Specifically, the harvesting control is defined as
\begin{equation}\label{eq:uData}
u_{\mathrm{obs}}(t)=\frac{F(t)}{F_{\mathrm{ref}}},
\end{equation}
where $F_{\mathrm{ref}}$ is the reference fishing mortality used to normalize the observed fishing mortality $F(t)$. 

The normalization of biomass and harvesting effort is used consistently in the numerical implementation. The units and interpretations of the parameters used in the simulations are summarized in Table~\ref{tab:Reducedparameters} and Table~\ref{tab:Spatialparameters}.

For the spatial simulations, we consider the one-dimensional habitat
$\Omega=(0,L)$ and discretize it using a finite difference method with spatial nodes $x_i=i\Delta x$, $i=0,\ldots,N$, and time levels $t_k=k\Delta t$. The homogeneous Neumann boundary conditions correspond to zero biomass flux at the habitat boundaries, ensuring that dispersal occurs only within the domain.

The state equation is solved forward in time, whereas the adjoint equation is solved backward from the terminal condition
$p(x,T)=0$. A semi-implicit discretization of the state equation is given by
\begin{equation}\label{eq:numerical_state}
\frac{B^{k+1}-B^k}{\Delta t}
=
D\Delta_h B^{k+1}
+
\frac{aB^k}{b+B^k}
-
\alpha u^k B^k
-
mB^k.
\end{equation}

The adjoint equation is discretized backward in time as
\begin{equation}\label{eq:numerical_adjoint}
-\frac{p^{k+1}-p^k}{\Delta t}
=
D\Delta_h p^k
+
\left(
\frac{ab}{(b+B^k)^2}
-
\alpha u^k
-
m
\right)p^k
+
\alpha u^k.
\end{equation}

In addition to biomass dynamics, the adjoint variable is used as a diagnostic quantity. While the biomass distribution $B(x,t)$ describes
the current state of the population, the adjoint variable $p(x,t)$ quantifies the marginal future value of biomass. Consequently, regions with large values of $p$ correspond to locations where biomass is particularly important for future system performance, whereas smaller
values indicate regions where harvesting has a weaker long-term impact.

\section{Numerical Simulations}\label{sec:numericalSi}

\subsection{Comparison with Observed Data}

The reduced model~\eqref{eq:reduced_model} is compared with the
available spawning stock biomass (SSB) observations to assess its ability
to reproduce the main temporal patterns of the observed biomass
dynamics. The harvesting input $u_{\mathrm{obs}}(t)$ is reconstructed from
the observed fishing pressure data and incorporated directly into the
model dynamics. 
The resulting $u_{\mathrm{obs}}(t)$ is then used as a prescribed input
to the reduced biomass model.

Model performance is assessed by comparing the simulated biomass
trajectory with the corresponding observed SSB data. As a quantitative
measure of agreement, we compute the root-mean-square error (RMSE),
\[
\mathrm{RMSE}
=
\left(
\frac{1}{N}
\sum_{i=1}^{N}
\left(
B(t_i)-B_i^{\mathrm{obs}}
\right)^2
\right)^{1/2},
\]
where $B(t_i)$ denotes the model prediction and
$B_i^{\mathrm{obs}}$ the corresponding observed biomass value at time
$t_i$.

The model parameters used in this comparison are summarized in
Table~\ref{tab:Reducedparameters}. These values are taken as
representative biological and harvesting parameters and are not estimated
through an optimization procedure using the observed biomass data.

\begin{table}[h]
\centering
\begin{tabular}{llll}
\toprule
Parameter & Meaning & Representative & Unit\\&&  value or range & \\
\midrule
$a$ & Maximum recruitment rate & $1.55$--$1.57$ & day$^{-1}$ \\
$b$ & Normalized saturation  & $10$--$10.068$ & dimensionless \\&parameter&&\\
$m$ & Natural mortality rate & $0.136$ & day$^{-1}$ \\
$\alpha$ & Harvesting-rate scaling  & $1$--$1.2$ & day$^{-1}$ \\&coefficient&&\\
$u_{\max}$ & Maximum admissible & $0.006$ & dimensionless \\&harvesting effort &&\\
\bottomrule
\end{tabular}
\caption{Representative model parameters used in the comparison with
observed biomass data. Biomass $B$ and harvesting effort $u$ are
expressed on normalized, dimensionless scales.}
\label{tab:Reducedparameters}
\end{table}

The purpose of this comparison is not to obtain an exact fit to the
observations, but rather to examine whether the reduced model can
reproduce the main temporal trends of the observed biomass dynamics
under the observed harvesting pressure. The resulting comparison
provides a data-informed assessment of the reduced model and motivates
its subsequent use in the spatial simulations presented below. The resulting RMSE over the analyzed 1974-2025 comparison period is \(\mathrm{RMSE}=0.28233\), indicating a moderate discrepancy between the simulated and observed biomass trajectories.

\begin{figure}[H]\label{fig:reduce}
\centering
\includegraphics[width=\textwidth]{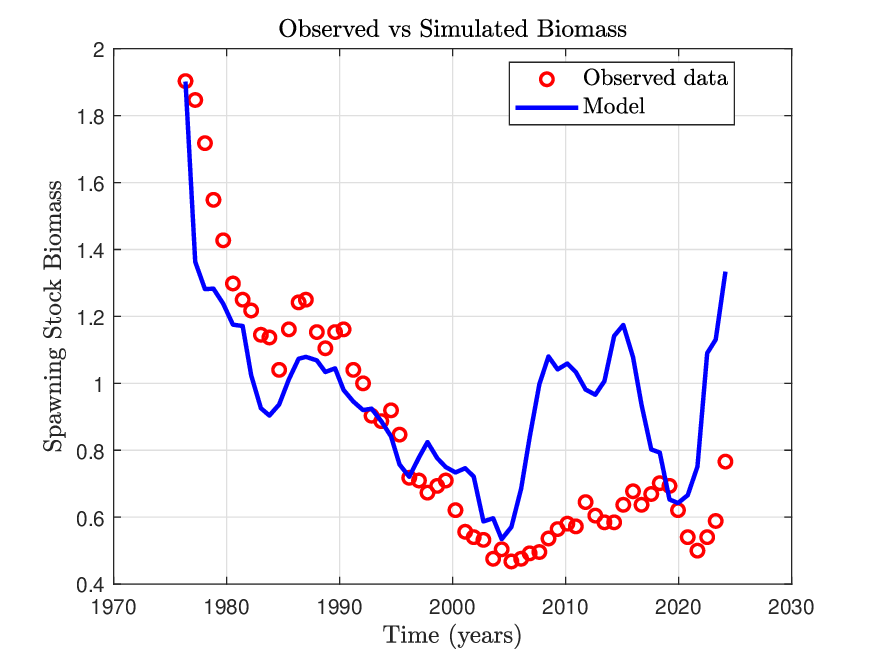}
\caption{Biomass dynamics obtained from the model (solid line) compared with observed data (markers). The observed data correspond to spawning stock biomass (SSB), defined as the total weight of mature individuals in the population and measured in tonnes over the reported time period 1974--2025 (see Figure~2 in the reference study \citep{slu_aqua2025}). The model curves are computed using parameter values within the ranges given in Table~\ref{tab:Reducedparameters}, illustrating the agreement obtained using representative parameter
values.}
\end{figure}

The reduced-model comparison provides the empirical motivation for the
subsequent spatial analysis. In particular, it establishes a biomass-based
description against which the effects of spatial movement and the future
management value of biomass can be examined.

\subsection{Spatially Distributed Optimal Harvesting}

The spatial model~\eqref{eq:spatial_model} is solved on a
one-dimensional domain $\Omega=(0,L)$ with an initially heterogeneous biomass distribution,
\[
B(x,0)
=
B_0
+
\varepsilon
\sin\!\left(\frac{2\pi x}{L}\right).
\]

The state--adjoint system is solved using a forward--backward sweep method.
The control is updated iteratively during these forward--backward sweeps.
The iteration is initialized with
\[
u^{(0)}(x,t)=0.1,
\]
which serves as the initial control over the space--time domain.

At iteration $k$, the state equation is first solved forward in time using
the current control $u^{(k)}(x,t)$, followed by the adjoint equation
backward in time. Based on the resulting state and adjoint variables, a
new control is then computed from the projection formula,
\[
u_{\mathrm{proj}}^{(k+1)}(x,t)
=
\operatorname{Proj}_{[0,u_{\max}]}
\left(
\frac{\alpha B^{(k)}(x,t)
\left(1-p^{(k)}(x,t)\right)}{\gamma}
\right).
\]
This is the new, unrelaxed control obtained directly from the optimality
condition. It is then combined with the current control $u^{(k)}$ through
the relaxation step,
\[
u^{(k+1)}(x,t)
=
r\,u^{(k)}(x,t)
+
(1-r)\,u_{\mathrm{proj}}^{(k+1)}(x,t),
\]
where $r\in[0,1]$ is the relaxation parameter. Thus, the relaxed control
$u^{(k+1)}$ becomes the current control in the next forward--backward
sweep, and the same projection formula is applied again using the newly
computed state and adjoint variables. In the numerical experiments,
$r=0.99$, so that 99\% of the current control is retained and 1\% of the
newly computed projected control is incorporated into the next iterate.

Convergence was monitored using the maximum pointwise change between two
successive relaxed control iterates,
\begin{equation}\label{eq:uConverge}
E_k =
\max_{x,t}
\left|u^{(k)}(x,t)-u^{(k-1)}(x,t)\right|.
\end{equation}
The iteration was stopped when
\begin{equation}\label{eq:uStop}
E_k < 10^{-6},
\end{equation}
or when the maximum number of 1000 iterations was reached.

The parameter values used in the spatial optimal-control simulation are
summarized in Table~\ref{tab:Spatialparameters}. The forward--backward
iteration converged after 652 iterations, with
\[
E_{652}=9.894185\times10^{-7}.
\]

The convergence behavior and the resulting spatial--temporal
distributions of the biomass, adjoint variable, and optimal harvesting
control are examined in the following subsection.

\begin{center}
\begin{table}[H]
\centering
\begin{tabular}{llll}
\toprule
Parameter & Meaning & Representative & Unit\\
&& value or range&  \\
\midrule
$a$ & Maximum recruitment rate & $1.55$ & day$^{-1}$ \\
$b$ & Normalized saturation & $10$ & dimensionless \\& parameter&&\\
$m$ & Natural mortality rate & $0.136$ & day$^{-1}$ \\
$\alpha$ & Harvesting-rate scaling & $1$ & day$^{-1}$ \\& coefficient&&\\
$u_0$ & Feedback harvesting scale & $0.1$ & dimensionless \\
$\gamma$ & Harvesting-cost coefficient & 3 & day$^{-1}$ \\
$u_{\max}$ & Maximum admissible & 0.2 & dimensionless \\& harvesting effort&&\\
$L$ & Spatial domain length &10 &spatial model  \\&&& unit\\
$D$ & Spatial diffusion coefficient & 0.01 & spatial unit$^2$/day \\
$T$ & Simulation time & 540 & days \\
$B_0$ & Baseline initial biomass & 1.2 & dimensionless \\
$\varepsilon$ & Initial heterogeneity &0.2 &- \\
\bottomrule
\end{tabular}
\caption{Representative model parameters, their interpretation, and units used in the numerical implementation. Biomass $B$ and harvesting effort $u$ are expressed on normalized, dimensionless scales.}
\label{tab:Spatialparameters}
\end{table}
\end{center}
\subsubsection{Numerical convergence and spatial--temporal structure of the optimal solution}
\label{subsec:strong_convergence}

The convergence of the forward--backward sweep method is assessed using
the maximum change in the optimal harvesting control between two
successive iterations, as defined in Equation~\eqref{eq:uConverge}.
For a relaxation parameter of $r=0.99$, the convergence error decreases
strongly as the iterations proceed, indicating that the numerical
optimal control approaches a stable solution. The small final
iteration error demonstrates that further updates to the harvesting
control become negligible, providing numerical evidence of convergence of the iterative scheme of the forward--backward sweep procedure.

To further examine the structure of the converged solution, the
spatial--temporal distributions of the biomass $B(x,t)$, adjoint
variable $p(x,t)$, and optimal harvesting control $u^*(x,t)$ are shown
in Figure~\ref{fig:three_figures}(a)--(c), respectively. The three
surfaces provide a complementary view of the converged optimality
system and illustrate how the spatial heterogeneity evolves over the
time horizon.

\begin{figure}[H]
    \centering

    \begin{minipage}[b]{0.49\textwidth}
        \centering
        \includegraphics[width=\textwidth]{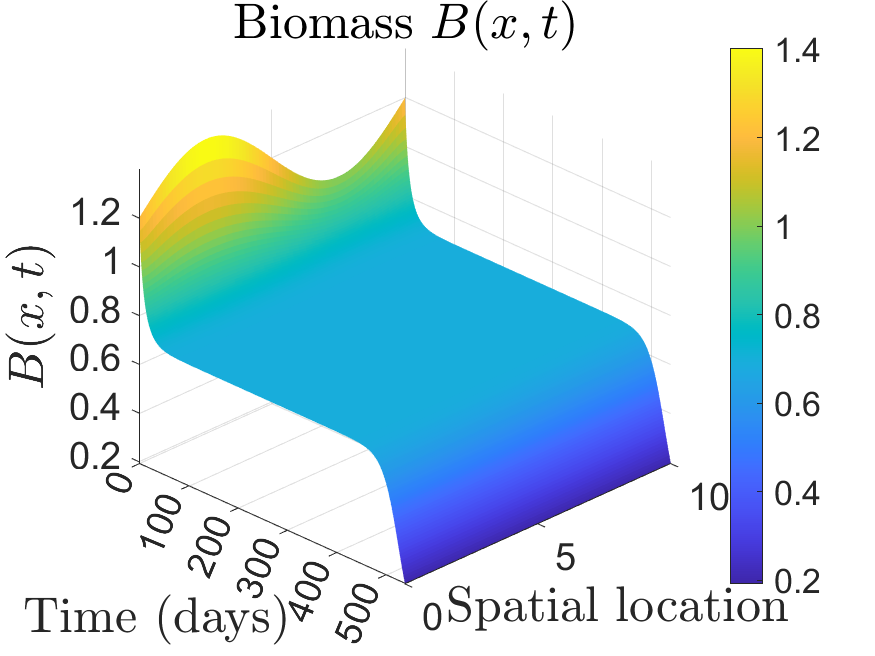}
        \vspace{2mm}
        
        (a)
    \end{minipage}
    \hfill
    \begin{minipage}[b]{0.49\textwidth}
        \centering
        \includegraphics[width=\textwidth]{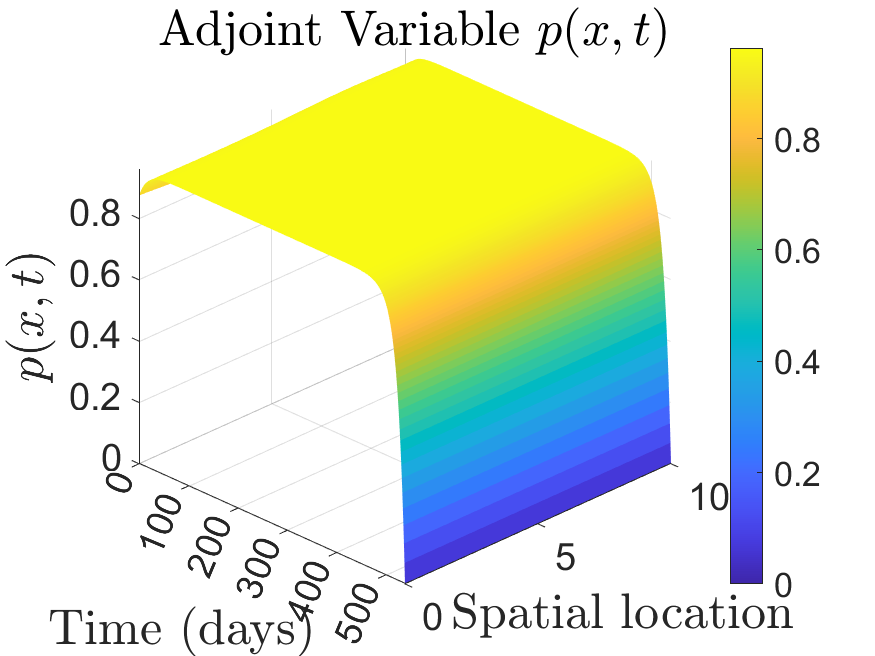}
        \vspace{2mm}
        
        (b)
    \end{minipage}

    \vspace{5mm}

    \begin{minipage}[b]{0.49\textwidth}
        \centering
        \includegraphics[width=\textwidth]{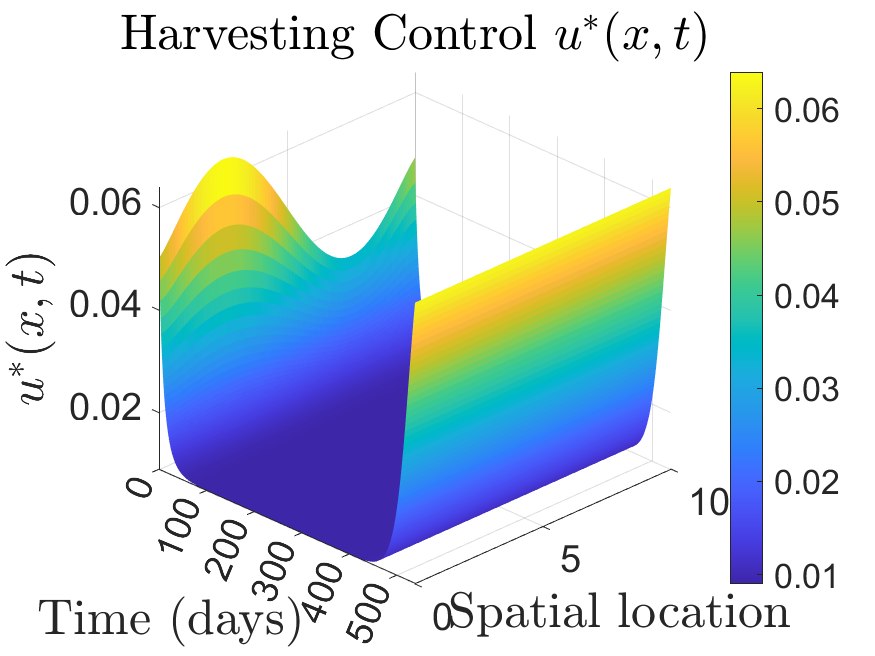}
        \vspace{2mm}
        
        (c)
    \end{minipage}

    \caption{Spatial--temporal structure of the converged optimality
    system for $r=0.99$ after the forward--backward sweep iterations:
    (a) biomass $B(x,t)$, (b) adjoint variable $p(x,t)$, and
    (c) optimal harvesting control $u^*(x,t)$.}
    \label{fig:three_figures}
\end{figure}

The biomass surface in Figure~\ref{fig:three_figures}(a) reflects the
initial spatial heterogeneity introduced through the initial biomass
distribution. As time progresses, this spatial variation is gradually
smoothed, and the biomass becomes increasingly homogeneous across the
spatial domain. This behavior is consistent with the diffusive
component of the state equation, which acts to reduce spatial
differences in the biomass distribution.

The adjoint variable shown in Figure~\ref{fig:three_figures}(b)
exhibits a corresponding spatially smooth structure. Its value remains
relatively high over most of the time horizon and decreases toward the
terminal time. In particular, the adjoint satisfies the terminal
condition
\[
p(x,T)=0,
\]
which accounts for the sharp transition observed near the terminal
time. Thus, this feature of the surface is associated with the
terminal condition of the optimality system rather than with a sudden
change in the spatial dynamics of the biomass.

The optimal harvesting control in Figure~\ref{fig:three_figures}(c)
is directly related to both the biomass and adjoint variables through
the optimality condition, Equation \eqref{eq:optimal_control_pde},
subject to the imposed admissible bounds. Consequently, regions and
periods with larger biomass contribute to a larger harvesting control,
whereas a larger adjoint value reduces the harvesting intensity. The
surface of $u^*(x,t)$ therefore reflects the combined influence of the
state and adjoint variables.

As the spatial heterogeneity of the biomass and adjoint variables is
reduced, the optimal harvesting control also becomes increasingly
homogeneous across the spatial domain. This provides a consistent
connection between the three components of the optimality system:
the diffusion-driven smoothing of the biomass, the corresponding
spatial structure of the adjoint variable, and the resulting spatial
distribution of the optimal harvesting control.

A sharp transition in the harvesting control is also visible near the
terminal time in Figure~\ref{fig:three_figures}(c). This behavior is
consistent with the terminal condition $p(x,T)=0$. At the final time,
the optimality condition consequently reduces to
\[
u^*(x,T)
=
\frac{\alpha B(x,T)}{\gamma},
\]
provided that the resulting value lies within the admissible control
interval. The terminal-time feature should therefore be interpreted
as a consequence of the imposed terminal condition rather than as a
separate biological transition.

Overall, the maximum iteration error provides the quantitative measure
of convergence, while the three-dimensional surfaces in
Figure~\ref{fig:three_figures}(a)--(c) illustrate the spatial--temporal
structure of the resulting converged solution. The progressive
reduction in spatial heterogeneity and the small final control update
together support the numerical convergence of the forward--backward
sweep method for the chosen relaxation parameter.

\subsubsection{Convergence of the forward--backward iteration}

To assess the numerical convergence of the forward--backward algorithm, we monitored the change in the control between two successive iterations.
Specifically, after the $n$-th forward--backward sweep, we defined the iteration error by Equation \eqref{eq:uConverge}.
Thus, $E_n$ measures the maximum pointwise change in the computed control over the spatial and temporal grid.
The forward--backward iteration was stopped when the Equation \eqref{eq:uStop} is valid.

The relaxation parameter was set to $\omega=0.99$ to stabilize the fixed-point iteration. The overall decay of the iteration error is shown in Table~\ref{tab:convergence}.

\begin{table}[htbp]
\centering
\caption{Convergence of the forward--backward iteration.}
\label{tab:convergence}
\begin{tabular}{r c}
\hline
Iteration $n$ & Maximum control change $E_n$ \\
\hline
1   & $1.000000\times 10^{-3}$ \\
10  & $9.135172\times 10^{-4}$ \\
100 & $3.697296\times 10^{-4}$ \\
200 & $5.078566\times 10^{-4}$ \\
400 & $1.927978\times 10^{-5}$ \\
500 & $5.669646\times 10^{-6}$ \\
600 & $1.780349\times 10^{-6}$ \\
652 & $9.894185\times 10^{-7}$ \\
\hline
\end{tabular}
\end{table}

The iteration error decreases substantially as the forward--backward sweeps proceed. At iteration $n=652$, the error satisfies
\[
E_{652}=9.894185\times10^{-7}<10^{-6}.
\]
Therefore, the prescribed numerical stopping criterion is satisfied, and the forward--backward iteration is terminated after 652 iterations.
This provides a numerical convergence check for the computed forward--backward stationary solution.

\subsubsection{Grid and time-step refinement}

To assess the numerical robustness of the spatial optimal-control simulation, a grid and time-step refinement test was performed. The baseline discretization used $N_x=80$ spatial intervals and $N_t=500$
time steps, corresponding to $\Delta x=0.125$ and $\Delta t=1.08$ days.
The resolution was then doubled to $N_x=160$ and $N_t=1000$, giving
$\Delta x=0.0625$ and $\Delta t=0.54$ days. All biological and optimal-control parameters, as well as the convergence criterion, were kept unchanged.

Both simulations satisfied the prescribed forward--backward iteration tolerance of $10^{-6}$. The mean final biomass changed from $0.18572$ to
$0.18835$, corresponding to a relative difference of approximately $1.40\%$. The corresponding mean final control changed by approximately $2.08\%$, while the objective functional changed by only $0.36\%$.
These relatively small changes indicate that the computed optimal-control solution is numerically stable under the tested spatial and temporal grid refinement.

\begin{table}[htbp]
\centering
\begin{tabular}{c c c c c c c}
\hline
$N_x$ & $N_t$ & $\Delta x$ & $\Delta t$ &
$\overline{B}(T)$ & $\overline{u}(T)$ & $J(u)$ \\
\hline
80  & 500  & 0.1250 & 1.08 & 0.18572 & 0.06089 & 40.58463 \\
160 & 1000 & 0.0625 & 0.54 & 0.18835 & 0.06218 & 40.73257 \\
\hline
\end{tabular}
\caption{Grid and time-step refinement for the spatial optimal-control simulation. Here, $T=540$ days denotes the final simulation time, $\overline{B}(T)$ and $\overline{u}(T)$ denote the spatially averaged biomass and harvesting control, respectively, evaluated at the final time, and $J(u)$ denotes the objective functional.}
\label{tab:refinement}
\end{table}

\subsection{Numerical Validation of Spatial Stability}

To numerically examine the spatial stability result derived in Section 5, we
first recall how the decay rate of the first nonzero spatial mode is obtained.
Under constant harvesting \(u_c\), the spatial model can be written as
\[
\partial_t B
=
D\Delta B
+
f(B),
\]
where
\[
f(B)
=
\frac{aB}{b+B}
-(m+\alpha u_c)B.
\]

The positive spatially homogeneous equilibrium \(B^*\) satisfies
\[
f(B^*)=0,
\]
and is given by
\[
B^*
=
\frac{a}{m+\alpha u_c}-b.
\]

We introduce a small perturbation around this equilibrium by writing
\begin{equation}\label{eq:Diff}
B(x,t)=B^*+z(x,t).
\end{equation}

Substituting this expression into the spatial model gives
\[
\partial_t z
=
D\Delta z+f(B^*+z).
\]
For a sufficiently small perturbation, the nonlinear term can be linearized
around \(B^*\):
\[
f(B^*+z)
\approx
f(B^*)+f'(B^*)z.
\]
Since \(B^*\) is an equilibrium, \(f(B^*)=0\), and therefore
\[
\partial_t z
=
D\Delta z+f'(B^*)z.
\]

The derivative of the reaction term is
\[
f'(B)
=
\frac{ab}{(b+B)^2}-(m+\alpha u_c),
\]
so that
\[
f'(B^*)
=
\frac{ab}{(b+B^*)^2}-(m+\alpha u_c).
\]

For the one-dimensional domain \(\Omega=(0,L)\) with homogeneous Neumann
boundary conditions, the spatial eigenfunctions are
\[
\cos\left(\frac{k\pi x}{L}\right),
\qquad k=0,1,2,\ldots,
\]
with corresponding Laplacian eigenvalues
\[
\mu_k=\left(\frac{k\pi}{L}\right)^2.
\]
We therefore consider a perturbation in the form
\begin{equation}\label{eq:pertub}
z(x,t)
=
c_k(t)\cos\left(\frac{k\pi x}{L}\right).
\end{equation}

Since
\[
\Delta z=-\mu_k z,
\]
the perturbation equation becomes
\begin{equation}\label{eq:BPertub}
\frac{dc_k}{dt}
=
\left[f'(B^*)-D\mu_k\right]c_k.
\end{equation}

Hence, the growth or decay rate of the \(k\)-th spatial mode is
\[
\lambda_k
=
f'(B^*)-D\mu_k.
\]
For the first nonzero spatial mode, \(k=1\), this gives
\[
\boxed{
\lambda_1
=
f'(B^*)-D\left(\frac{\pi}{L}\right)^2
}.
\]

We now evaluate this expression using the parameter values in
Table~\ref{tab:Spatialparameters} and the constant harvesting effort
\(u_c=0.01\). 
Using the parameter values in Table~\ref{tab:Spatialparameters} and $u_c
= 0.01$, we obtain $\lambda_1<0$.

Thus, the first nonzero spatial mode satisfies
\[
\frac{dc_1}{dt}=\lambda_1c_1,
\]
and hence
\[
c_1(t)=c_1(0)e^{\lambda_1t}.
\]

The corresponding analytical perturbation, Equation \eqref{eq:pertub}, is 
\[
z_{\mathrm{an}}(x,t)
=
c_1(0)e^{\lambda_1t}
\cos\left(\frac{\pi x}{L}\right).
\]

For comparison with the numerical solution, we measure the perturbation
amplitude by
\[
A(t)
=
\max_{x\in[0,L]}|B(x,t)-B^*|
\]
where 
\[
B(x,0)
=
B^*
+
c_1(0)\cos\left(\frac{\pi x}{L}\right),
\qquad
c_1(0)=0.05.
\]
at the initial condition.

For the analytical solution, this amplitude is
\[
\begin{aligned}
A_{\mathrm{an}}(t)
&=
\max_{x\in[0,L]}
|z_{\mathrm{an}}(x,t)|\\
&=
|c_1(0)|e^{\lambda_1t},
\end{aligned}
\]
since
\[
\max_{x\in[0,L]}
\left|
\cos\left(\frac{\pi x}{L}\right)
\right|
=1.
\]

The spatial reaction--diffusion equation is solved numerically using the
same finite-difference discretization and time-stepping approach described
above. The numerical perturbation amplitude \(A(t)\) is then compared with
the analytical decay curve \(A_{\mathrm{an}}(t)\). 

The numerical perturbation amplitude is compared with this analytical
exponential decay curve in Figure~\ref{fig:stability_validation}. The two
curves are in close agreement, showing that the numerical perturbation
decays at a rate consistent with the analytical prediction. This provides
numerical support for the spatial stability result derived in Section 5 and
demonstrates the damping of spatial perturbations toward the homogeneous
equilibrium under the parameter regime considered.

\begin{figure}[H]
    \centering
    \includegraphics[width=0.75\textwidth]{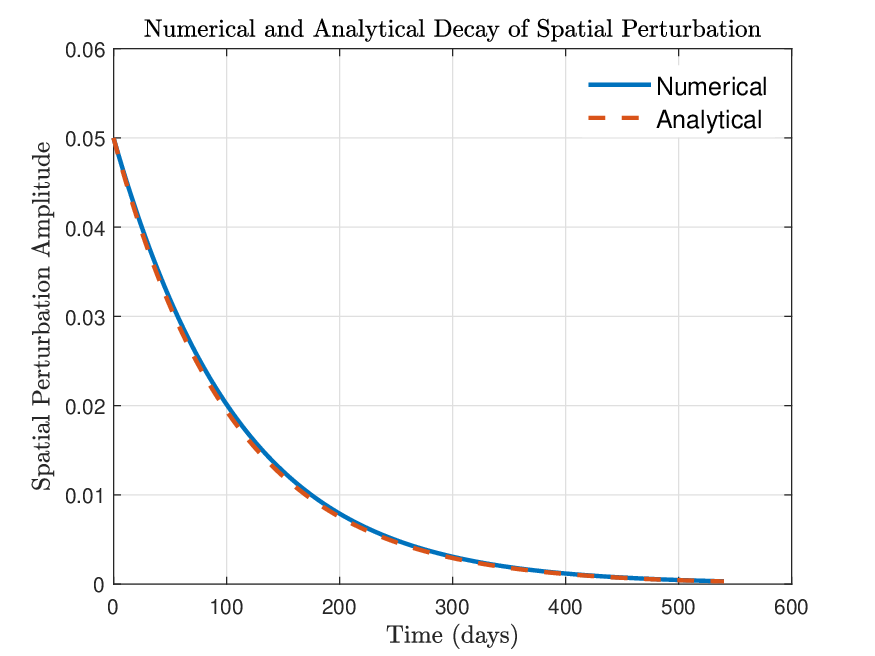}
    \caption{Comparison between the numerical decay of the spatial perturbation and the analytical exponential decay predicted by the first nonzero spatial eigenvalue \(\lambda_1\).}
    \label{fig:stability_validation}
\end{figure}

\subsection{Parameter Sensitivity and Numerical Robustness}
To assess the robustness of the numerical results, a one-at-a-time
sensitivity analysis was performed by varying the diffusion coefficient
$D$ while keeping all other model parameters and numerical settings fixed.
The baseline simulation uses $D=0.01$.
For the sensitivity analysis, the additional diffusion coefficients
\[
D=0,\;0.1,\;0.5,\;1,\;2,\;5
\]
were tested. 
For each tested value of $D$, the forward--backward sweep procedure was
repeated using the same convergence criterion as in the baseline
simulation.

The sensitivity analysis focuses on the mean final biomass and the objective functional, which provide biological and management measures of the resulting optimal solution. 
The dependence of the objective functional and mean final biomass on the diffusion coefficient is shown in the below figure.

\begin{figure}[htbp]
    \centering

    \begin{minipage}{0.49\textwidth}
        \centering
        \includegraphics[width=\textwidth]{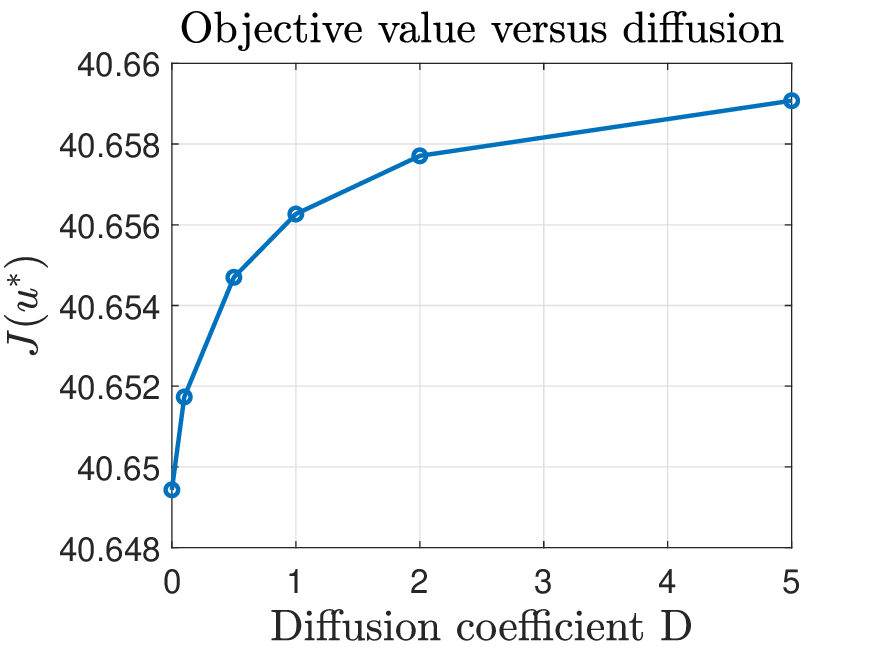}
        
        \small (a) Objective functional $J(u^*)$ versus diffusion coefficient $D$.
    \end{minipage}
    \hfill
    \begin{minipage}{0.49\textwidth}
        \centering
        \includegraphics[width=\textwidth]{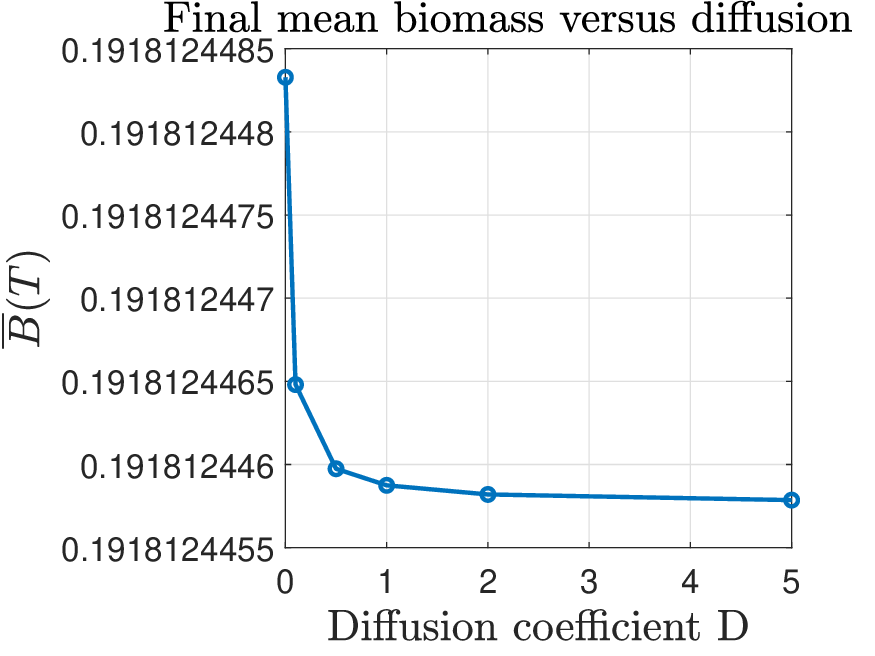}
        
        \small (b) Mean final biomass $\overline{B}(T)$ versus diffusion coefficient $D$.
    \end{minipage}

    \caption{Sensitivity of the optimal-control solution to variation in the diffusion coefficient $D$.}
    \label{fig:diffusion_sensitivity}
\end{figure}

The results show that both quantities vary only modestly over the tested range of diffusion coefficients. In particular, the mean final biomass remains nearly unchanged, while the objective functional exhibits only a small variation as \(D\) increases. These results indicate that the principal numerical outcomes of the optimal-control problem are robust to the tested variation in the diffusion coefficient.

The weak sensitivity of the aggregate quantities does not imply that diffusion has no effect on the spatial dynamics. Rather, under the parameter regime and time horizon considered, diffusion primarily influences the spatial redistribution and homogenization of biomass, while its effect on the final aggregate biomass and overall harvesting objective remains limited.

\subsubsection{Relaxation-Parameter Convergence Diagnostics}\label{sec:sensitivityRelax}

The purpose of this analysis is to assess the effect of the relaxation
parameter $r$ on the convergence behavior of the forward--backward
optimization algorithm. The relaxation parameter determines how strongly the control from
the previous iteration is retained when updating the control in the
forward--backward sweep. Therefore, different values of $r$ may lead to
different rates of change of the control during the iterative
procedure. To examine this effect, the relaxation parameter is varied
from $0.10$ to $0.99$ in increments of $0.01$. For each value of $r$,
the forward--backward sweep is performed for the same number of
iterations, namely 10,000, using the same spatial and temporal
discretization and model parameters.

The change in the control between two successive iterations is measured
using the convergence quantity defined in
Equation~\eqref{eq:uConverge},
\[
E_k =
\max_{x,t}
\left|u^{(k)}(x,t)-u^{(k-1)}(x,t)\right|.
\]
Thus, at the final iteration $k=10000$, $E_{10000}(r)$ provides the
maximum absolute change in the control between the last two iterations
for a given value of the relaxation parameter. Since the same number of
iterations is used for all values of $r$, the quantity
$E_{10000}(r)$ provides a direct comparison of the residual control
change after the prescribed number of iterations. Smaller values of
$E_{10000}(r)$ indicate that the control has undergone a smaller change
between the final two iterations and is therefore closer to a
numerically stationary state.

The resulting values of $E_{10000}(r)$ are presented in
Table~\ref{tab:relaxation_sensitivity}. The results show how the
relaxation parameter influences the convergence behavior of the
forward--backward iteration over the fixed 10,000-iteration interval.

\begin{table}[H]
\centering
\begin{tabular}{cc}
\toprule
$r$ & $E_{10000}(r)$ \\
\midrule
0.50 & 0.071073 \\
0.55 & 0.037495 \\
0.60 & 0.036014 \\
0.65 & 0.023664 \\
0.70 & 0.023247 \\
0.75 & 0.015253 \\
0.80 & 0.0098588 \\
0.85 & 0.0068305 \\
0.90 & 0.0015352 \\
0.95 & $2.7756\times10^{-17}$ \\
0.99 & $6.9389\times10^{-18}$ \\
\bottomrule
\end{tabular}
\caption{Algorithmic convergence diagnostics for selected values of the
relaxation parameter $r$. The quantity $E_{10000}(r)$ denotes the maximum
absolute change in the control between the final two iterations after
10,000 forward--backward iterations.}
\label{tab:relaxation_sensitivity}
\end{table}

The spatial--temporal distributions of the biomass $B(x,t)$, the adjoint
variable $p(x,t)$, and the optimal harvesting control $u^*(x,t)$ are
compared for $r=0.60$, $0.70$, $0.90$, and $0.99$ after 1000
forward--backward iterations. The heatmaps illustrate the spatial
structure of the numerical solutions for these representative
relaxation parameters. The solutions become progressively less
spatially heterogeneous as $r$ increases, while the corresponding
convergence measure $E_{10000}(r)$ reported in Table~\ref{tab:relaxation_sensitivity}
also decreases substantially for larger values of $r$.

\begin{figure}[H]
    \centering
    \includegraphics[width=0.75\textwidth]{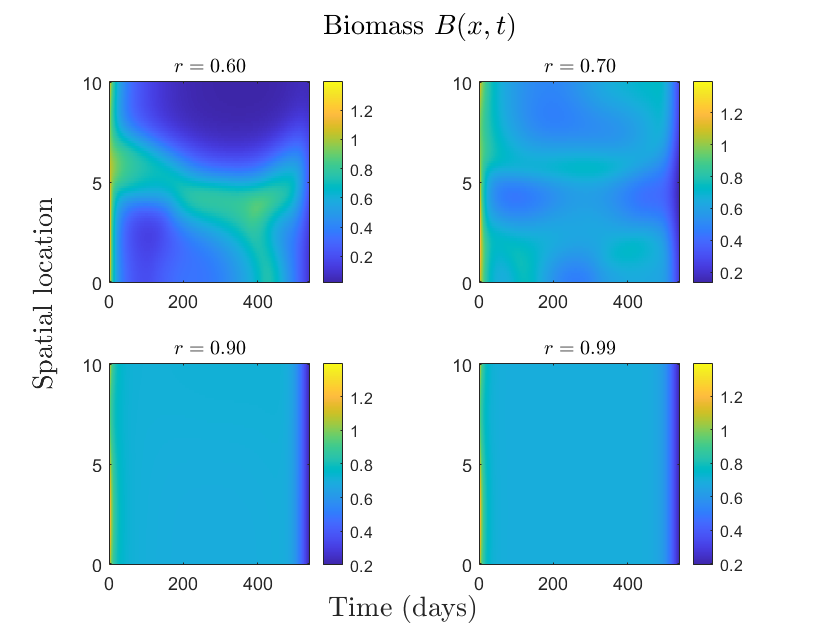}
    \caption{Spatial–temporal distribution of the biomass \(B(x,t)\) after 1000 iterations, illustrating the differences in the degree of numerical convergence after a fixed number of iterations for different relaxation parameters.}
    \label{fig:controlRelaxB}
\end{figure}

\begin{figure}[H]
    \centering
    \includegraphics[width=0.75\textwidth]{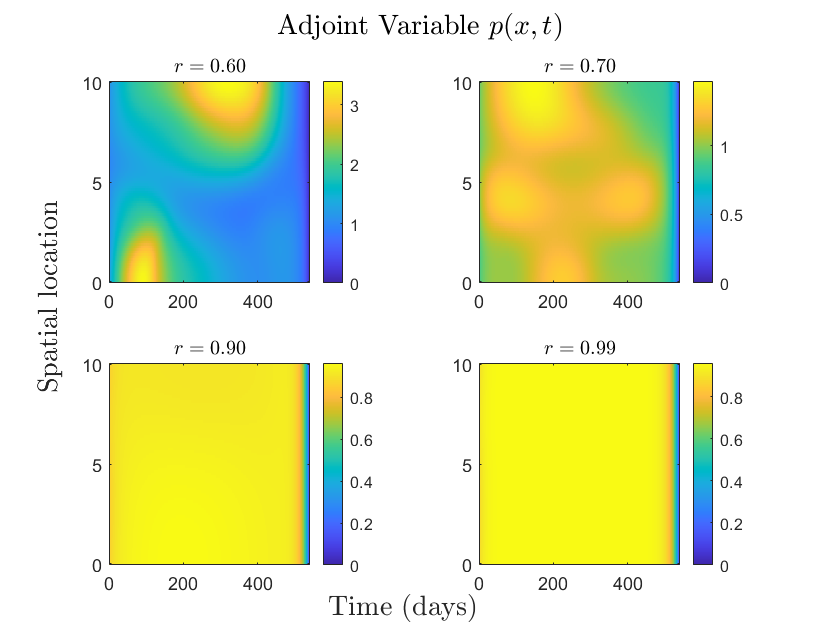}
    \caption{Spatial–temporal distribution of the adjoint variable \(p(x,t)\) after 1000 iterations, illustrating the differences in the degree of numerical convergence after a fixed number of iterations for different relaxation parameters.}
    \label{fig:controlRelaxp}
\end{figure}

\begin{figure}[H]
    \centering
    \includegraphics[width=0.75\textwidth]{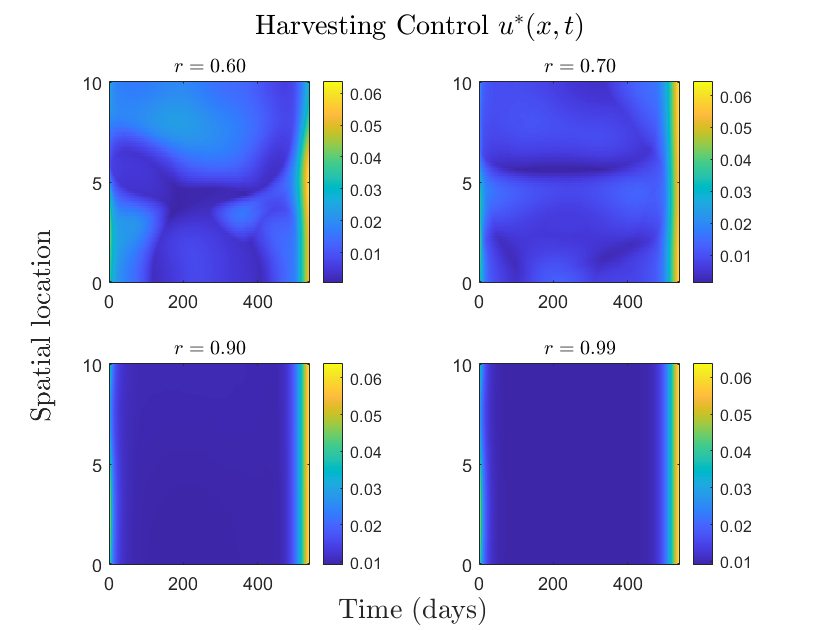}
    \caption{Spatial–temporal distribution of the optimal harvesting control \(u^*(x,t)\) after 1000 iterations, illustrating the differences in the degree of numerical convergence after a fixed number of iterations for different relaxation parameters.}
    \label{fig:controlRelaxu}
\end{figure}

\section{Discussion and Conclusion}\label{sec:discussion}

The proposed framework provides a unified way to examine the interaction between spatial biomass dynamics, marginal future biomass value, and harvesting effort. The reduced biomass model provides a tractable description of density-dependent production, natural mortality, and harvesting, while the spatial extension incorporates biomass redistribution and spatially distributed harvesting. The comparison with spawning stock biomass (SSB) observations provides a data-informed assessment \citep{slu_aqua2025} of the temporal behavior of the reduced model and supports its use as a basis for the spatial analysis.

The analytical results establish positivity, boundedness, and well-posedness of the biomass dynamics. For the reduced model, the equilibrium analysis identifies a critical harvesting threshold separating persistence and extinction regimes under constant harvesting. The spatial stability analysis further shows that diffusion contributes to the decay of small spatial perturbations around the positive spatially homogeneous equilibrium. These results provide the mathematical basis for the subsequent optimal-control formulation.

The adjoint variable provides information that complements the current biomass state. While biomass describes the present population distribution, the adjoint represents the marginal future value of biomass within the harvesting objective. This information enters directly into the optimal harvesting rule,
\[
u^*(x,t)
=
\operatorname{Proj}_{[0,u_{\max}]}
\left(
\frac{\alpha B^*(x,t)(1-p(x,t))}{\gamma}
\right).
\]
Consequently, the harvesting decision depends jointly on current biomass and its marginal future value. Biomass abundance alone therefore does not fully determine the harvesting decision within the proposed optimal-control framework.

The numerical simulations illustrate this coupled state--adjoint--control structure. Diffusion progressively reduces the initial spatial variation in biomass, while the computed adjoint and optimal harvesting fields exhibit limited spatial variation under the parameter regime considered. Thus, the present numerical results primarily demonstrate the interaction between biomass dynamics, future management value, and harvesting effort within the proposed framework. 

The parameter sensitivity analysis further indicates that the main aggregate numerical outcomes are robust to the tested variation in the diffusion coefficient. Although changes in \(D\) affect the rate of spatial redistribution and homogenization, the mean final biomass and objective functional remain relatively stable over the tested range. This supports the robustness of the principal numerical findings under the considered parameter regime.

The forward--backward sweep procedure yields a stationary solution satisfying the derived first-order optimality conditions \citep{mcasey2012convergence}. The convergence and grid-refinement tests show that the computed solution is stable under the tested numerical settings. These results provide confidence in the numerical implementation and in the reliability of the computed solution.

The present formulation adopts a simplified spatial setting that provides a clear framework for examining the interaction between biomass dynamics, future management value, and harvesting decisions. The spatial model uses a one-dimensional domain and spatially homogeneous biological parameters, while diffusion provides a simple representation of spatial redistribution. Similarly, the reduced biomass model focuses on overall biomass dynamics rather than explicitly representing age structure or environmental variability \citep{quinn1999quantitative}. The empirical comparison is data-informed and is intended to illustrate the temporal behavior of the reduced model rather than provide a full statistical calibration with parameter uncertainty. These choices keep the present analysis focused on the mathematical and optimal-control structure of the model and provide a useful foundation for more detailed extensions.

Future work could extend the framework to spatially heterogeneous biological parameters, multi-patch habitats, spatially resolved observations, or age-structured population models \citep{webb1985theory,murray2002mathematical}. Such extensions would allow the effects of spatial differences in production, mortality, and movement on the adjoint value and harvesting decisions to be investigated in greater detail.
\section*{Acknowledgements}
We acknowledge SLU Aqua, Department of Aquatic Resources, for providing the fisheries information available through Fiskbarometern.

\bibliography{bib}

@misc{slu_aqua2025,
  author = {{Institutionen f{\"o}r akvatiska resurser (SLU Aqua)}},
  title  = {Art{\"o}versikt 2025},
  year   = {2025},
  url    = {https://fiskbarometern.se/rapport/2025/species/Sill\%2FStr\%C3\%B6mming},
  note   = {Accessed 2026}
}

@Book{cantrell2004spatial,
  author    = {Cantrell, R. S. and Cosner, C.},
  publisher = {John Wiley \& Sons},
  title     = {Spatial ecology via reaction-diffusion equations},
  year      = {2004},
}

@Article{murray2002mathematical,
  author  = {Murray, J. D.},
  journal = {Mathematical Biology, Springer},
  title   = {Mathematical biology: I. An introduction. Interdisciplinary applied mathematics},
  year    = {2002},
  volume  = {17},
}

@Book{lions1971optimal,
  author    = {Lions, J. L.},
  publisher = {Springer},
  title     = {Optimal control of systems governed by partial differential equations},
  year      = {1971},
  volume    = {170},
}

@Book{lenhart2007optimal,
  author    = {Lenhart, S. and Workman, J. T.},
  publisher = {Chapman and Hall/CRC},
  title     = {Optimal control applied to biological models},
  year      = {2007},
}

@Book{webb1985theory,
  author    = {Webb, G. F.},
  publisher = {CRC Press},
  title     = {Theory of nonlinear age-dependent population dynamics},
  year      = {1985},
}

@Book{clark2010mathematical,
  author    = {Clark, C. W.},
  publisher = {John Wiley \& Sons},
  title     = {Mathematical bioeconomics: the mathematics of conservation},
  year      = {2010},
}

@Book{evans2022partial,
  author    = {Evans, L. C.},
  publisher = {American mathematical society},
  title     = {Partial differential equations},
  year      = {2022},
  volume    = {19},
}

@Book{hilborn2013quantitative,
  author    = {Hilborn, R. and Walters, C. J.},
  publisher = {Springer Science \& Business Media},
  title     = {Quantitative fisheries stock assessment: choice, dynamics and uncertainty},
  year      = {2013},
}

@Book{quinn1999quantitative,
  author    = {Quinn, T. J. and Deriso, R. B.},
  publisher = {oxford university Press},
  title     = {Quantitative fish dynamics},
  year      = {1999},
}

@Book{okubo2001diffusion,
  author    = {Okubo, A. and Levin, S. A. and others},
  publisher = {Springer},
  title     = {Diffusion and ecological problems: modern perspectives},
  year      = {2001},
  volume    = {14},
}

@Article{goethel2021revival,
  author    = {Goethel, D. R. and Cadrin, S. X.},
  journal   = {ICES Journal of Marine Science},
  title     = {Revival and recent advancements in the spatial fishery models originally conceived by Sidney Holt and Ray Beverton},
  year      = {2021},
  number    = {6},
  pages     = {2298--2315},
  volume    = {78},
  publisher = {Oxford University Press},
}

@Article{goethel2011incorporating,
  author    = {Goethel, D. R. and Quinn, T. J. and Cadrin, S. X.},
  journal   = {Reviews in Fisheries Science},
  title     = {Incorporating spatial structure in stock assessment: movement modeling in marine fish population dynamics},
  year      = {2011},
  number    = {2},
  pages     = {119--136},
  volume    = {19},
  publisher = {Taylor \& Francis},
}

@Article{ding2009optimal,
  author    = {Ding, W. and Lenhart, S.},
  journal   = {Natural Resource Modeling},
  title     = {Optimal harvesting of a spatially explicit fishery model},
  year      = {2009},
  number    = {2},
  pages     = {173--211},
  volume    = {22},
  publisher = {Wiley Online Library},
}

@Article{cohen1987review,
  author    = {Cohen, Y.},
  journal   = {Canadian Journal of Fisheries and Aquatic Sciences},
  title     = {A review of harvest theory and applications of optimal control theory in fisheries management},
  year      = {1987},
  number    = {S2},
  pages     = {s75--s83},
  volume    = {44},
  publisher = {NRC Research Press Ottawa, Canada},
}

@Article{kelly2016optimal,
  author    = {Kelly J., Michael R. and Xing, Y. and Lenhart, S.},
  journal   = {Natural Resource Modeling},
  title     = {Optimal fish harvesting for a population modeled by a nonlinear parabolic partial differential equation},
  year      = {2016},
  number    = {1},
  pages     = {36--70},
  volume    = {29},
  publisher = {Wiley Online Library},
}

@Article{mcasey2012convergence,
  author    = {McAsey, M. and Mou, L. and Han, W.},
  journal   = {Computational Optimization and Applications},
  title     = {Convergence of the forward-backward sweep method in optimal control},
  year      = {2012},
  number    = {1},
  pages     = {207--226},
  volume    = {53},
  publisher = {Springer},
}

@Article{soukaina2022optimal,
  author    = {Soukaina, B. R. and Imane, A. and Mostafa, R. and Naceur, A. and Youssef, E. F.},
  journal   = {Chaos, Solitons \& Fractals},
  title     = {Optimal control of a phytoplankton-zooplankton spatiotemporal discrete bioeconomic model},
  year      = {2022},
  pages     = {112020},
  volume    = {158},
  publisher = {Elsevier},
}

\end{document}